\documentclass[a4 paper,11pt]{article}
\usepackage{stmaryrd}
\usepackage{amsfonts}
\usepackage{bbm}
\usepackage{amscd}
\usepackage{mathrsfs}
\usepackage{latexsym,amssymb,amsmath,amscd,amscd,amsthm,amsxtra,xypic}
\usepackage[dvips]{graphicx}
\usepackage[utf8]{inputenc}
\usepackage[T1]{fontenc}
\usepackage{enumerate}
\usepackage{lmodern}
\usepackage{amssymb}
\usepackage[all]{xy}
\usepackage{ifpdf}
\ifpdf
\usepackage[colorlinks=true,linkcolor=blue,final,backref=page,hyperindex,citecolor=red]{hyperref}
\else
\usepackage[colorlinks,final,backref=page,hyperindex,hypertex]{hyperref}
\fi

\usepackage{nicefrac,mathtools,enumitem}
\usepackage{microtype}
\usepackage{amsfonts}
\usepackage{amssymb}
\usepackage{mathrsfs}
\usepackage{amsmath}
\usepackage{amsthm}
\usepackage{enumerate}
\usepackage{indentfirst}
\usepackage{amsfonts}
\usepackage{amssymb}
\usepackage{mathrsfs}
\usepackage{amsmath}
\usepackage{amsthm}
\usepackage{enumerate}
\usepackage{cite}
\usepackage{mathrsfs}
\usepackage{geometry}
\allowdisplaybreaks

\newtheorem{thm}{Theorem}[section]
\newtheorem{lem}[thm]{Lemma}
\newtheorem{cor}[thm]{Corollary}
\newtheorem{pro}[thm]{Proposition}
\newtheorem{ex}[thm]{Example}
\newtheorem{exs}[thm]{Examples}

\newtheorem{rmk}[thm]{Remark}
\newtheorem{defi}[thm]{Definition}

\numberwithin{equation}{section}

\newcommand{\be }{\begin{equation}}
\newcommand{\ee }{\end{equation}}

\newcommand{\br}[1]{   [ \cdot,    \cdot  ]   }

\newcommand{\Hom}{\mathrm{Hom}}

\newcommand{\rprod}{\dashv}  
\newcommand{\lprod}{\vdash}  
\newcommand{\mprod}{\bot}   

\newcommand{\K}{\mathbb{K}}

\newcommand {\emptycomment}[1]{}

\def\<{\langle}
\def\>{\rangle}
\def\a{\alpha}
\def\b{\beta}

\def\c{\cdot}

\def\o{\otimes}

\def\r{\pi}

\def\o{\circ}

\begin{document}

\title{\bf  Hom-associative algebras, Admissibility and Relative averaging operators}
\date{}
\author{\normalsize \bf  S. Braiek\small{$^{1}$} \footnote {   E-mail: safabraiek25@gmail.com}, T. Chtioui\small{$^{2}$} \footnote{   E-mail: chtioui.taoufik@yahoo.fr},  M. Elhamdadi\small{$^{3}$} \footnote {  E-mail: emohamed@math.usf.edu (Corresponding author)} and  S. Mabrouk\small{$^{4}$} \footnote{  E-mail: mabrouksami00@yahoo.fr}}
\date{{\small{$^{1}$ Faculty of Sciences, University of Sfax,   BP
1171, 3000 Sfax, Tunisia. \\
\small{$^{2}$ University of Gabes, Faculty of Sciences Gabes, City Riadh 6072 Zrig, Gabes, Tunisia \\
\small{$^{3}$} Department of Mathematics,
University of South Florida, Tampa, FL 33620, U.S.A.\\  \small{$^{4}$} Faculty of Sciences, University of Gafsa,   BP
2100, Gafsa, Tunisia.
 }}}}

\maketitle

\begin{abstract}
 We introduce the notion of relative averaging operators on Hom-associative algebras with a representation.  Relative averaging operators are twisted generalizations of
relative averaging operators on associative algebras. We give two characterizations of relative  averaging operators of  Hom-associative algebras via graphs and Nijenhuis operators.  A (homomorphic) relative  averaging operator of  Hom-associative algebras with respect to a given representation gives rise to  Hom-associative (tri)dialgebras. By admissibility, a Hom-Jordan (tri)dialgebra and a Hom-(tri)Leibniz algebra can be obtained from Hom-associative (tri)dialgebra.\\

\end{abstract}


\tableofcontents


\section{Introduction}

Algebras where the identities defining the structure are twisted by homomorphisms
are called Hom-algebras. They have been intensively investigated in the literature recently.
The theory of Hom-algebra started from Hom-Lie algebras introduced and discussed in
\cite{Hartwig,Larson}, motivated by quasi-deformations of Lie algebras of vector fields, in particular q-deformations of Witt and Virasoro algebras. Hom-associative algebras were introduced in
\cite{MakSil}  while Hom-Jordan algebras were introduced in \cite{Makhlouf,Yau} as twisted
generalizations of Jordan algebras.

The notion of averaging operator was first implicitly considered in \cite{rey} by O. Reynolds in 1895. His description was given in terms of idempotent Reynolds operators that appeared in the turbulence theory of fluid dynamics. In 1930, Kamp\'{e} de F\'{e}riet gave the explicit description of an averaging operator in the context of turbulence theory and functional analysis \cite{kampe}.  Let $A$ be an associative algebra. A linear map $P: A \rightarrow A$ is said to be an averaging operator on the algebra $A$ if
\begin{align*}
    P(a)  P(b) = P ( P(a)  b) = P (a  P(b)), \text{ for } a, b \in A.
\end{align*}

An associative algebra $A$ equipped with a distinguished averaging operator on it is called an averaging algebra. In the last century, averaging operators and averaging algebras were mostly studied in the algebra of functions on a space \cite{birkhoff,kelly,gam-mill,mill,moy}. In the algebraic contexts, B. Brainerd \cite{brain} first studied the conditions for which an averaging operator can be realized as a generalization of the integral operator on the ring of real-valued measurable functions. In his PhD thesis, W. Cao \cite{cao} considered averaging operators from algebraic and combinatorial points of view. Among others, he described the Lie algebra and the Leibniz algebra induced from an averaging operator and constructed the free commutative averaging algebra. Further algebraic study of averaging operators on any binary operad and their relations with bisuccessors, duplicators and Rota-Baxter operators are explicitly given in \cite{bai-imrn,pei-rep,pei-rep2}. Recently, J. Pei and L. Guo \cite{pei-guo} constructed the free (associative) averaging algebra using a class of bracketed words, called averaging words (see also \cite{das, DasMandal}). It is worth mentioning that averaging operators also appeared in the context of Lie algebras with the name of embedding tensors. They have close connections with Leibniz algebras, tensor hierarchies and higher gauge theories \cite{bon,lavau,strobl,sheng-embedd}.

The notion of embedding tensors (relative averaging operators) \cite{Nicolai} can be traced back to the study of the 
gauge supergravity theory. Embedding tensors were used to construct $N =8$ supersymmetric gauge theories and
to study the Bagger-Lambert theory of multiple $M2$-branes in \cite{Berg2}. Embedding tensors were also
used to study deformations of maximal $9$-dimensional supergravity \cite{Fernandez}. Recently, this
topic has attracted much attention of the mathematical physics world \cite{Kotov,lavau}.

Our main objective is to introduce the notion of relative averaging operators on  Hom-associative algebras with a representation.  
These are twisted generalizations of
relative averaging operators on associative algebras. We give two characterizations of  averaging operators of  Hom-associative algebras via graphs and Nijenhuis operators.  A (homomorphic) relative  averaging operator of  Hom-associative algebras with respect to a given representation gives rise to  Hom-associative (tri)dialgebras and by admissibility, a Hom-Jordan (tri)dialgebra and Hom-(tri)Leibniz algebra can be obtained by Hom-associative (tri)dialgebra.

This paper is organized as follow, In Section \ref{Preliminaries}, we recall some knwon results and definitions on Hom-associative, Hom-Lie, Hom-Leibniz and Hom-Jordan algebras and their representations and Hom-associative dialgebras. Section \ref{Duplication} is devoted to  study the notion of a relative averaging operator on a Hom-associative, Hom-Lie and Hom-Jordan algebra with respect to a suitable  bimodule. We give a characterization of  a relative averaging operator by graphs and Nijenhuis operators. We investigate relative averaging operators on Hom-$\Omega$ algebras to obtain
Hom-di-$\Omega$ algebras and present some basic observations.
  In Section  \ref{triplication}, we introduce the notion of a homomorphic relative   averaging operator  on a Hom-$\Omega$ algebra which allows us the replicating of  the multiplication into three multiplications of the same nature.

All   vector spaces, linear and multilinear maps, unadorned tensor products and wedge products
are over a field $\mathbb K$ of characteristic $0$.

\section{Review  of Hom-associative algebras and some related structures}\label{Preliminaries}
In this section we recall some knwon results and definitions of Hom-associative, Hom-Lie, Hom-Leibniz and Hom-Jordan algebras and their representations and Hom-associative dialgebras. For more details, we refer to \cite{Hartwig, Larson, Makhlouf, MakSil, Yau}.

\begin{defi}
A Hom-associative algebra is a triple $(A,\cdot,\alpha)$ consisting of a vector space  $A$ togother  with a binary operation $\cdot:A \otimes A \to A$ and a linear map $\alpha: A \to A$ satisfying, for any $x,y,z\in A$,
$$ (x,y,z)_\a \equiv  (x \cdot y)\cdot \alpha(z)-\alpha(x) \cdot (y \cdot z)=0.   $$
If $\a=id$, we recover the classical associative algebras.
\end{defi}
If  the twist map $\alpha$ preserves the operation, then the Hom-associative algebra is called  multiplicative.
For this reason, throughout this paper all  Hom-algebras are considered multiplicative.
Let $(A_1,\cdot_1,\alpha_1)$ and $(A_2,\cdot_2,\alpha_2)$ be two Hom-associative algebras. A linear map $\phi:A_1\rightarrow A_2$ is a morphism of Hom-associative algebras if
$$\phi\circ \alpha_1=\alpha_2\circ \phi \ \ \ \ \text{and}\ \ \ \  \phi(x \cdot_1 y)=\phi(x) \cdot_2 \phi(y),\ \forall \ x,y\in A.  $$

\begin{defi}\label{RepHomAss}
 A bimodule (or a representation) of a Hom-associative algebra $(A,\cdot, \alpha)$ on a vector space $V$ consists of a  linear map  $\beta\in End(V)$ and two linear maps  $l, r: A\rightarrow End(V)$
such that
\begin{align}
\beta l(x)=l(\alpha(x))\beta, \quad
 \beta r(x)=& r(\alpha(x))\beta\label{rAs1},\\
 l(x \cdot y)\beta=l(\alpha(x))l(y), \quad
 r(x \cdot y)\beta=& r(\alpha(y))r(x), \quad
 l(\alpha(x))r(y)=r(\alpha(y))l(x)\label{rAs4}
\end{align}
 for any $x, y \in A.$
 A Hom-associative algebra $(A,\c,\a)$ with a representation $(V,l,r,\b)$ is a called \textsf{HAss-Rep} pair
and refer to it with the pair $(A,V)$.
\end{defi}
Let $(A,\c_A,\a_A)$ and $(B,\c_B,\a_B)$ two Hom-associative algebras. We will say that $A$  acts on $B$ if there are tow $\mathbb{K}$-linear maps
$l,r: A \to End(B)$
such that
 $(B,l,r,\a_B)$ is a bimodule of $A$, and for any $x\in A$ and $b,b^\prime\in B$ we have
 \begin{align}
&l(\a_A(x))(b\c_Bb^\prime)=(l(x)b)\c_B(\a_B(b^\prime)),\quad r(\a_A(x))(b\c_Bb^\prime)=\a_B(b)\c_B(r(x)b^\prime)),\quad\nonumber\\\label{actass}&\a_B(b)\c_B(l(x)b^\prime)=(r(x)b)\c_B(\a_B(b^\prime)).
 \end{align}
An action of $A$ on $B$ is denoted by $(B,l,r,\a_B,\c_B)$. A \textsf{HAss-Act} pair  is a Hom-associative algebra $(A,\c,\a)$ with an action $(B,l,r,\a_B,\c_B)$
and refer to it with the pair $(A,B)$.
\begin{lem}
    Note that a tuple $(B,l,r,\alpha_B,\cdot_B)$ is an action  of  a Hom-associative algebra $(A,\cdot_A,\alpha_A)$
if and only if  $A \oplus V$  carries a Hom-associative structure with product and twisting map  given by
\begin{align*}
&(x + u) \cdot_{A \oplus B}  (y + v )=  x \cdot_A  y + l(x)v + r(y )u +u \cdot_B  v, \\
&(\alpha_A+\alpha_B)(x+u)=\alpha_A(x)+\alpha_B(u)
\end{align*}
 for all  $x,y \in A$ and $ u,v \in B,$ which  is called the semi-direct product of $A$ with $B$.
\end{lem}
In the following, we recall the notion of Hom-Jordan and Hom-Lie algebras. These two varieties are closely related to Hom-associative algebras,\emptycomment{ in the sense that any Hom-associative algebra is Hom-Jordan and Hom-Lie admissible}
(see \cite{MakSil} and \cite{Makhlouf} for more details). In fact, let $(A,\cdot,\alpha)$ be a Hom-associative algebra. Define the algebras $A^-:=(A,[\cdot,\cdot],\alpha)$ and $A^+:=(A,\o,\alpha)$ where
\begin{align*}
    & [x,y]=x\cdot y-y\cdot x\ (\text{commutator}) \\
    & x\o y=x\cdot y+y\cdot x\  (\text{anti-commutator})
\end{align*}

\begin{thm}
Under the following notations, $A^+$ is a Hom-Jordan algebra and $A^-$ is a Hom-Lie algebra. In other words,
Hom-associative algebras are  Hom-Jordan-admissible and Hom-Lie-admissible.
\end{thm}

We aim to extend these relationships to the level of dialgebraic and trialgebraic settings.
\begin{defi}
  A Hom-Jordan algebra is a triple $(A, \o, \alpha)$
    consisting of a  vector space $A$ and two linear maps
    $\alpha: A  \to A$ and
    $\o: A \otimes A \to A$ such that, for any $x,y \in A$
    \begin{align}
        x \o y=& y \o x, \\
      \alpha(x^2)\o  (\alpha(y) \o \alpha(x))= &(x^2 \o \alpha(y))\o \alpha^2(x). \label{HOm-Jor id}
  \end{align}
\end{defi}

 \begin{defi}
A representation (or a module) of a Hom-Jordan algebra $(A,\o,\a)$ on a vector space $V$ with respect to $\b \in gl(V)$
is a linear map $\r: A \to gl(V)$ such that
 for any $x,y,z \in A$
\begin{align}\label{representation}
& \b  \r(x)= \r(\a(x))  \b , \\
&\r(\a^2(x))\r(y\circ z)  \b + \r(\a^2(y))\r(z\circ x)  \b  + \r(\a^2(z))\r(x\circ y)  \b  \nonumber \\
=& \r(\a(x)\circ \a(y))\r(\a(z)) \b  +  \r(\a(y)\circ \a(z))\r(\a(x)) \b +  \r(\a(z)\circ \a(x))\r(\a(y)) \b, \label{RepHomJor1}\\
& \r((x\circ y)\circ \a(z)) \b^2+ \r(\a^2(x))\r(\a(z))\r(y)+\r(\a^2(y))\r(\a(z))\r(x) \nonumber \\
 =& \r(\a(x)\circ \a(y))\r(\a(z)) \b +  \r(\a(y)\circ \a(z))\r(\a(x)) \b +  \r(\a(x)\circ \a(z))\r(\a(y)) \b.\label{RepHomJor2}
\end{align}

 \end{defi}
A Hom-Jordan algebra $(A,\o,\a)$ with a representation $(V,\pi,\b)$ is called \textsf{HJor-Rep} pair
and refer to it with the pair $(A,V)$.
 \begin{defi}
A Hom-Lie algebra is a triple $(A,[\cdot,\cdot],\alpha)$
   consisting of a  vector space $A$ and two linear maps
    $\alpha: A  \to A$ and
    $[\cdot,\cdot]: A  \wedge A \to A$ such that, for any $x,y, z \in A$
    \begin{align}
& [\alpha(x),[y,z]] + [\alpha(y),[z,x]]+ [\alpha(z),[x,y]] =0, \quad \textrm{(Hom-Jacobi identity)}.
  \end{align}
 \end{defi}
\begin{defi}
    A representation (or a module) of a Hom-Lie algebra $(A,[\cdot,\cdot],\alpha)$ on a vector space $V$, with respect to $\b \in gl(V)$, is a linear map $\rho: A \to gl(V)$, such that for all $x, y \in A$, the following
identities  are satisfied:
\begin{align}
    \rho(\a(x))\b= & \b \rho(x),\\
    \rho([x,y])\b=& \rho(\a(x))\rho(y)-\rho(\a(y))\rho(x).
\end{align}

\end{defi}
A Hom-Lie algebra $(A,[\c,\c],\a)$ with a representation $(V,\rho,\b)$ is called \textsf{HLie-Rep} pair
and refer to it with the pair $(A,V)$.
   Let $(A,[\c,\c]_A,\a_A)$ and $(B,[\c,\c]_B,\a_B)$ be two Hom-Lie algebras. We will say that $A$ acts on $B$ if there is a linear map $\rho:A\to End(B)$ such that $(B,\rho,\a_B)$ is a module of $A$, and for any $x\in A$ and $u,v\in B$ we have
    \begin{equation}\label{actLie}
        \rho(\a_A(x))[u,v]_B=[\rho(x)u,\a_B(v)]_B+[\a_B(u),\rho(x)v]_B.
    \end{equation}
    A \textsf{HLie-Act} pair is a Hom-Lie algebra $(A,[\c,\c],\a)$ with an action $(B,\rho,\a_B,[\c,\c]_B)$
and refer to it with the pair $(A,B)$.
\begin{lem}
   The  tuple $(B,\rho,\alpha_B,[\cdot,\cdot]_B)$ is an action  of  a Hom-Lie  algebra $(A,\alpha_A,[\cdot,\cdot]_A)$
if and only if  $A \oplus B$  carries a Hom-Lie structure with bracket and twisting map  given by
\begin{align*}
&[x + u,  y + v ]_{A \oplus B}=  [x ,  y  ]_{A } + \rho(x)v -\rho(y)u +[u ,   v ]_{ B}, \\
&(\alpha_A+\alpha_B)(x+u)=\alpha_A(x)+\alpha_B(u)
\end{align*}
 for all  $x,y \in A$ and $ u,v \in B,$ which  is called the semi-direct product of $A$ with $B$.
\end{lem}

\begin{defi}
    A Hom-Leibniz algebra is a tuple $(A,\{\cdot,\cdot\},\a)$ consisting of a vector space $A$ endowed with a two linear maps $\a: A \to A$ and $\{\c,\c\}: A \otimes A \to A$ such that for any $x,y,z \in A$
    \begin{align}
        \{\a(x),\{y,z\}\}=\{\{x,y\},\a(z)\}+\{\a(y),\{x,z\}\},
\quad \textrm{(Hom-Leibniz identity).}
\end{align}
\end{defi}

A Hom-dialgebra is a vector space $A$  with two bilinear
operations $\vdash, \dashv: A \times A \to A$, called the left and right products and a linear map $\a: A \to A$.
On the other hand,  a Hom-$0$-dialgebra is a Hom-dialgebra $(A,\vdash,\dashv,\alpha)$ satisfying the left
and right bar identities:
\begin{align}
  (x \dashv y) \vdash \a(z)=
  ( x \vdash y ) \vdash \a(z),
  \qquad
  \a(x) \dashv ( y \dashv z )=
  \a(x) \dashv ( y \vdash z ).
\end{align}
Let $(A,\vdash,\dashv,\alpha)$ be a Hom-$0$-dialgebra. Define the left, right and inner Hom-associators as follow
\begin{align*}
&(x,y,z)^\a_\dashv=(x\dashv y)\dashv \a(z)-\a(x)\dashv(y \dashv z),\\
&(x,y,z)^\a_\vdash =(x\vdash y)\vdash \a(z)-\a(x)\vdash(y \vdash z),\\
&(x,y,z)^\a_\times =(x\vdash y)\dashv \a(z)-\a(x)\vdash(y \dashv z).
\end{align*}
In addition, define, for all $x,y\in A,$ the dicommutator by
\begin{align}
  \{x,y\}=x \vdash y-y \dashv x   , \label{dicomm}
\end{align}
and the anti-dicommutator  by
\begin{align}
 x \bullet y= x \vdash y+ y \dashv x .  \label{antidicomm}
\end{align}
  \begin{defi}\label{defdiass}
     A   Hom-associative dialgebra (also called Hom-diassociative algebra)
  is a Hom-$0$-dialgebra    $(A,\vdash,\dashv,\alpha)$ satisfying
left, right and inner Hom-associativity (for all $x,y,z\in A$):
\begin{align}\label{dialgebra}
(x,y,z)^\a_\dashv=0,\quad (x,y,z)^\a_\vdash=0,\quad (x,y,z)^\a_\times =0.
\end{align}
\end{defi}
Let $(A,\lprod,\rprod,\a)$ and $(A^\prime,\lprod^\prime,\rprod^\prime,\a^\prime)$ be two Hom-associative dialgebras. A morphism of Hom-associative dialgebras from $A$ to $A^\prime$ is a linear map $\phi:A\to A^\prime$ satisfying
\begin{equation*}
    \phi\a=\a^\prime\phi,\quad \phi(x\lprod y)=\phi(x)\lprod^\prime\phi(y),\quad \phi(x\rprod y)=\phi(x)\rprod^\prime\phi(y).
\end{equation*}
\begin{exs}\label{ex} We give the following examples:
\begin{enumerate}
    \item  If $(A,\cdot,\a)$ is a Hom-associative algebra, then the formulas $x \dashv y=x \cdot y=x \vdash y$ define a structure of Hom-associative dialgebra on $A$.
    \item Let $(A,\cdot,\a)$ be a Hom-associative algebra and $d: A \to A$ be a differential (that is $d(x\cdot y)=d(x)\c y+x \c d(y)$ and $d^2=0$).
    Define left and right products on $A$ by the formulas
    $$ x \dashv y:= x\c d(y),\quad x \vdash y:= d(x)\c y.$$
    It is immediate to check that $A$ equipped with these two products is a Hom-associative dialgebra.
  \item Let $(A,\c,\a)$ be a Hom-associative algebra and let $(V,l,r,\b)$ be an $A$-bimodule. Let $f : V \to A$ be an $A$-bimodule map, that is $$f(l(x)u)=x\cdot f(u) \quad \text{and}\quad f(r(x)u)=f(u)\cdot x,\quad \forall\ x\in A, u\in V.$$  Then one can put a dialgebra structure on $V$ as follows:
  $$ u \dashv u':=  r(f(u'))u ,\quad u \vdash u':=l(f(u))u'.     $$

  \item Let $(A,\vdash,\dashv,\a)$ be a Hom-associative dialgebra. Then the operations $x\dashv' y=y \vdash x$ and $x \vdash' y= y \dashv x$ define a new Hom-associative dialgebra structure on $A$.

\end{enumerate}
\end{exs}

\begin{pro}\label{asstoLeibniz}
   If $(A,\lprod,\rprod,\alpha)$ is a Hom-associative dialgebra, then $(A,\{\cdot,\cdot\},\a)$ is a Hom-Leibniz algebra.
\end{pro}

In the following, we recall the notion of Jordan dialgebras and Jordan trialgebras which are obtained by duplicating and triplicating of Jordan algebra.
Such algebras are investigated in \cite{Kolesnikov}.  Note that  a Jordan dialgebra is also named a Leibniz-Jordan algebra since it is related to a Jordan  algebra in the same way as Leibniz algebras relate to Lie algebras.

\begin{defi}
A (left) Jordan dialgebra is a vector space $A$ equipped with a binary operation $\bullet$ satisfying
\begin{align}
    (x\bullet y)\bullet z=& (y\bullet x)\bullet z,\\
   (x^2,y,z)=& 2(x,y,x\bullet z),\\
    x\bullet (x^2 \bullet y)=& x^2 \bullet (x \bullet y),
\end{align}
for any $x,y,z \in A$, where $x^2=x \bullet x$ and $(x,y,z)=(x \bullet y)\bullet z-x \bullet (y \bullet z)$.
\end{defi}
These equations are equivalent to the following polynomial identities:
\begin{align*}
& (x_1\bullet x_2)\bullet x_3 = (x_2\bullet x_1)\bullet x_3,\\
 & ((x_4\bullet  x_3)\bullet x_2)\bullet x_1
+ x_4\bullet (x_2\bullet (x_3\bullet  x_1))
+ x_3\bullet (x_2\bullet (x_4\bullet  x_1))  \\
 = & (x_4\bullet  x_3)\bullet(x_2\bullet  x_1)
 + (x_4\bullet  x_2)\bullet(x_3\bullet  x_1)
 + (x_3\bullet  x_2)\bullet(x_4\bullet  x_1) , \\
 & x_1\bullet ((x_4\bullet x_3)\bullet x_2)
+ x_4\bullet ((x_3\bullet  x_1)\bullet x_2)
+ x_3\bullet ((x_4\bullet x_1)\bullet x_2)   \\
 = &(x_4\bullet  x_3)\bullet (x_1\bullet x_2)
 + (x_1\bullet x_3)\bullet (x_4 \bullet x_2)
 + (x_4\bullet x_1)\bullet(x_3\bullet x_2).
\end{align*}

\section{Duplications via relative averaging operators}\label{Duplication}
Let $\Omega$ be a variety of algebras over $\mathbb K$  with a single bilinear operation
denoted by $\mu$. In this section, we investigate relative averaging operators on Hom-$\Omega$ algebras to obtain
Hom-di-$\Omega$ algebras and present some basic observations. Especially, we will treat the cases
of Hom-Lie, Hom-associative and  Hom-Jordan algebras.
We will treat each case separately.
Let $(A,\mu,\a)$ be a Hom-$\Omega$ algebra and $T: A \to A$ be a linear map. Then $T$ is called an averaging operator if it satisfies the following conditions:
\begin{align}
    & T \a=\a T,\\
    & T(\mu(T(x),y))=\mu(T(x),T(y))=T(\mu(x,T(y))),\ \forall x,y \in A.
\end{align}
The notion of Nijenhuis operators plays a crucial role in our study.  Recall that a linear map $ N: A \rightarrow A$ is said to be a Nijenhuis operator on the Hom-$\Omega$ algebra $A$ if for all $x,y \in A,$
\begin{align*}
& N\a= \a  N,  \\
& \mu(N(x) , N(y)) = N \big( \mu(N(x), y) ~+~ \mu( x , N(y)) - N\mu( x ,y) \big).
\end{align*}
  \subsection{The Hom-associative case}
  In what follows, we introduce the notion of a relative averaging operator on a Hom-associative algebra with respect to a bimodule.
  Let $(A,V)$ be a \textsf{HAss-Rep} pair as defined in Definition \ref{RepHomAss}.
  \begin{defi}
      A linear map $\mathcal K: V \to A$ is called
      \begin{enumerate}
          \item left relative averaging operator on $A$ with respect to $(V,l,r,\b)$ if
          \begin{align}
  &\mathcal K \b =\a \mathcal K\quad\text{and}\quad \mathcal K(u)\cdot \mathcal K(v)=\mathcal K(l(\mathcal K(u))v),\ \forall\ u,v \in V.
      \end{align}
      \item right relative averaging operator on $A$ with respect to $(V,l,r,\b)$ if
      \begin{align}
  &\mathcal K \b =\a \mathcal K\quad\text{and}\quad \mathcal K(u)\cdot \mathcal K(v)=\mathcal K( r(\mathcal K(v))u),\ \forall\ u,v \in V.
      \end{align}
      \item relative averaging operator on $A$ with respect to $(V,l,r,\b)$ if it's both left and right relative averaging operator.
      \end{enumerate}
  \end{defi}
  We give the following examples:
  \begin{ex}
Let $(A,\c,\a)$ be a Hom-associative algebra. Then the tensor product $A \otimes A$ can be equipped with an $A$-bimodule structure $(l,r,\b) $ with the left and right $A$-actions respectively given by
\begin{align*}
&l(x) (a \otimes b) = x \cdot a \otimes b,\quad r(x)(a \otimes b)  = a \otimes b \cdot x ~~~ \text{ and } ~~~ \b(a\otimes b)=\a(a)\otimes\a(b),\\ &\text{ for }  a \otimes b \in A \otimes A, \ x \in A.
\end{align*}
Consider the map $\mathcal{K}: A \otimes A \rightarrow A$ given by $\mathcal{K}(a \otimes b) = a \cdot b$, for $a \otimes b \in A \otimes A$. For any $a \otimes b,~ a' \otimes b' \in A \otimes A$, we have
\begin{align*}
\mathcal{K} (a \otimes b) \cdot \mathcal{K} (a' \otimes b') = a \cdot b \cdot a' \cdot b' = \begin{cases}
= \mathcal{K} \big(  a \cdot b \cdot a' \otimes b' \big) = \mathcal{K} \big(l( \mathcal{K} (a \otimes b))  (a' \otimes b') \big),\\
=\mathcal{K} \big( a \otimes b \cdot a' \cdot b'  \big) = \mathcal{K} \big( r(\mathcal{K} (a' \otimes b')) (a \otimes b)   \big).
\end{cases}
\end{align*}
This shows that $\mathcal{K} : A \otimes A \rightarrow A$ is a relative averaging operator.
\end{ex}
\begin{ex}
Let $(A,\c,\a)$ be a Hom-associative algebra. Then the space $B=\underbrace{A \oplus \cdots \oplus A}_{n \text{ copies}}$ is an $A$-bimodule where the left (resp. right) $A$-action is given componentwise left (resp. right) multiplication map and the morphism $\b$ is given by $$\b(a_1,\ldots,a_n)=(\a(a_1),\ldots,\a(a_n)).$$ Then it is easy to see that the map
\begin{align*}
\mathcal{K} : B \rightarrow A, ~~ \mathcal{K}\big((a_1, \ldots, a_n)\big) = a_1 + \cdots + a_n, \text{ for } (a_1, \ldots, a_n) \in B
\end{align*}
is a relative averaging operator.
\end{ex}
\emptycomment{\begin{ex}
Let $(A,\cdot,\a)$ be a Hom-associative algebra. Then for any $1 \leq i \leq n$, the $i$-th projection map $P_i : A \oplus \cdots \oplus A \rightarrow A$, $(a_1, \ldots, a_n) \mapsto a_i$ is a relative averaging operator.
\end{ex}}
\begin{ex}
     Let $(A,V)$ be a \textsf{HAss-Rep} pair. Let $f : V \to A$ be an $A$-bimodule map (defined in Example \ref{ex}, item 3). Then it is easy to see that $f$ is a relative averaging operator.
\end{ex}
 Define on the direct sum $A\oplus V$ the mutiplications $\lprod_{A\oplus V},\rprod_{A\oplus V}$ and the linear map $\alpha_{A\oplus V}$ by
\begin{eqnarray}
\alpha_{A\oplus V}(x+u)&=&\alpha(x)+\beta(u),\\
    (x+u)\lprod_{A\oplus V} (y+v)&=&x\cdot y+l(x)v,\\
    (x+u)\rprod_{A\oplus V} (y+v)&=&x\cdot y+r(y)u,
    \, \forall\, x,y\in A,\, u,v\in V.
\end{eqnarray}
 We first prove the following result.
\begin{pro}\label{hsdirect diass}
    With the above notations $(A\oplus V,\lprod_{A\oplus V},\rprod_{A\oplus V},\alpha_{A\oplus V})$ is a Hom-associative dialgebra, called \textbf{hemisemi-direct product} Hom-associative dialgebra and denoted by $A\oplus_{Diass} V$.
\end{pro}
\begin{proof}
For any $x,y,z\in A$ and $u,v,w\in V$ we have
    \begin{align*}
        &((x+u)\rprod_{A\oplus V}(y+v))\lprod_{A\oplus V}\alpha_{A\oplus V}(z+w)-\alpha_{A\oplus V}(x+u)\lprod_{A\oplus V}((y+v)\lprod_{A\oplus V}(z+w))\\=&(x\cdot y)\cdot \a(z)+l(x\cdot y)\beta(w)-\a(x)\cdot(y\cdot z)-l(\a(x))(l(y)w)\\=&0
        \end{align*}
        and
        \begin{align*}
        &\a_{A\oplus V}(x+u)\rprod_{A\oplus V}((y+v)\rprod_{A\oplus V}(z+w))-\a_{A\oplus V}(x+u)\rprod_{A\oplus V}((y+v)\lprod_{A\oplus V}(z+w))\\=&\a(x)\cdot(y\cdot z)+r(y\cdot z)\beta(u)-(x\cdot y)\cdot \a(z)-r(y\cdot z)\beta(u)\\=&0.
    \end{align*}
    Then, $(A\oplus V,\lprod_{A\oplus V},\rprod_{A\oplus V},\alpha_{A\oplus V})$ is a Hom-$0$-dialgebra. Similarly, we can check \eqref{dialgebra}.
Thus $(A\oplus V,\lprod_{A\oplus V},\rprod_{A\oplus V},\alpha_{A\oplus V})$ is a Hom-associative dialgebra.
\end{proof}
In the following, we give a characterization of a relative averaging operator in terms of graphs.
\begin{thm}\label{graphHass}
Let $(A,V)$ be a \textsf{HAss-Rep} pair. A linear map ${\mathcal K}:V\to A$ is a relative averaging operator on the Hom-associative algebra $(A,\cdot,\alpha)$ if and only if  the graph $Gr({\mathcal K}) = \{{\mathcal K}u+u\ |\ u \in V\}$ is a  subalgebra of the hemisemi-direct product $A\oplus_{Diass} V$.
\end{thm}
\begin{proof}
Let $\mathcal{K}: V \to A$ be a linear map.  For all $u,v\in V$, we have
\begin{align*}
(\mathcal{K}u+u)\lprod_{A\oplus V}(\mathcal{K}v+v)=&\mathcal{K}u\cdot \mathcal{K}v +l(\mathcal{K}u)v, \\
(\mathcal{K}u+u)\rprod_{A\oplus V}(\mathcal{K}v+v)=&\mathcal{K}u\cdot \mathcal{K}v +r(\mathcal{K}v)u.
\end{align*}
Therefore, the graph $Gr({\mathcal K}) = \{{\mathcal K}u+u\ |\ u \in V\}$ is a  subalgebra of the hemisemi-direct product $A\oplus_{Diass} V$ if and only if $\mathcal{K}$ satisfies $\mathcal{K}(u)\cdot \mathcal{K}(v)=\mathcal{K}(l(\mathcal{K}u)v)=\mathcal{K}(r(\mathcal{K}v)u)$, wich implies that $\mathcal{K}$ is a relative averaging operator on Hom-associative algebra $A$ with respect to the representation $(V,l,r,\beta)$.
\end{proof}

We have the following proposition whose proof is straightforward.
\begin{pro}
    Let $(A,V)$ be a \textsf{HAss-Rep} pair. A linear map $\mathcal K: V \to A$ is a relative averaging operator on $V$ over $A$ if and only if the map
$$N_{\mathcal K}: (A \oplus V) \to (A \oplus V),\ x+u \mapsto \mathcal K(u)$$
is a Nijenhuis operator on the  hemisemi-direct product $A \oplus_{Diass} V$.
\end{pro}
Note  that an averaging  operator on a Hom-associative algebra $A$ can be regarded as a relative averaging  operator on the adjoint $A$-bimodule $A$.
In the next result, we show that a relative averaging operator induces an averaging
operator.
\begin{pro}\label{rel av to av}
Let $(A,V)$ be a \textsf{HAss-Rep} pair and let $\mathcal K: V \to A$ be a relative averaging operator on $V$ over the algebra $A$. Then $\overline{\mathcal K}: A \oplus V \to A \oplus V,\ x+u \mapsto \mathcal K(u)+u $ is a left relative averaging operator (respectively right relative averaging operator) on the Hom-associative algebra $(A\oplus V,\lprod_{A\oplus V},\a_{A\oplus V})$, (respectively on $(A\oplus V,\rprod_{A\oplus V},\a_{A\oplus V})$) .
\end{pro}

\begin{proof}
According to Proposition \ref{hsdirect diass}, we have each of $\lprod_{A\oplus V}$ and $\rprod_{A\oplus V}$ is Hom-associative.
For any $x,y\in A$ and $u,v\in V$, we have
\begin{eqnarray*}
    \overline{\mathcal K}(x+u)\cdot_{A\oplus V} \overline{\mathcal K}(y+v)&=& (\mathcal{K}(u)+u)\cdot_{A\oplus V} (\mathcal{K}(v)+v)\\&=&(\mathcal{K}(u)\cdot\mathcal{K}(v))+l(\mathcal{K}(u))v\\&=&\mathcal{K}(l(\mathcal{K}(u)v))+l(\mathcal{K}(u))v.
\end{eqnarray*}
On the other hand, we have
\begin{eqnarray*}
    \overline{\mathcal K}(\overline{\mathcal K}(x+u)\cdot_{A\oplus V} (y+v))&=&\overline{\mathcal K}((\mathcal{K}(u)+u)\cdot_{A\oplus V}(y+v))\\&=&\overline{\mathcal K}(\mathcal{K}(u)\cdot y+l(\mathcal{K}(u))v)\\&=&\mathcal{K}(l(\mathcal{K}(u))v))+l(\mathcal{K}(u))v).
\end{eqnarray*}
Thus, $\overline{\mathcal{K}}$ is a left relative averaging operator on the Hom-associative algebra $(A\oplus V,\lprod_{A\oplus V},\a_{A\oplus V})$.
Similarly, $\overline{\mathcal{K}}$ is a right relative averaging operator on the Hom-associative algebra $(A\oplus V,\rprod_{A\oplus V},\a_{A\oplus V})$.
\end{proof}
\begin{thm}\label{asstodiass}
     Let $\mathcal K: V\rightarrow A$ be a relative averaging operator  on a \textsf{HAss-Rep} pair $(A,V)$.    Then there exists a Hom-associative dialgebra structure $(\vdash_\mathcal{K},\dashv_\mathcal{K})$ on $V$ given by
    \begin{align}
   &     u\vdash_{\mathcal K} v =l(\mathcal K(u))v,\quad
       u\dashv_{\mathcal K} v= r(\mathcal K(v))u,\ \forall u,v \in V.
    \end{align}
\end{thm}
\begin{proof}
First we begin by  showing  that $(V,\vdash_{\mathcal K},\b)$ and $(V,\dashv_{\mathcal K},\b)$ are Hom-associative algebras, let $u,v,w\in V$ we have
\begin{align*}
    (u\vdash_{\mathcal K}v)\vdash_{\mathcal K}\beta(w)-\beta(u)\vdash_{\mathcal K}(v\vdash_{\mathcal K} w)=&l({\mathcal K}(l({\mathcal K}u)v))\beta(w)-l({\mathcal K}\beta(u))(l({\mathcal K}v)w)\\
    =&l({\mathcal K}u\cdot{\mathcal K}v)\beta(w)-l(\a({\mathcal K}u))(l({\mathcal K}v)w)\\
    =&0,
 \end{align*}
 and
    \begin{align*}
    (u\dashv_{\mathcal K}v)\dashv_{\mathcal K}\beta(w)-\beta(u)\dashv_{\mathcal K}(v\dashv_{\mathcal K}w)=&r({\mathcal K}\beta(w))(r({\mathcal K}v)u)-r({\mathcal K}(r({\mathcal K}w)v))\beta(u)\\
    =&r(\a({\mathcal K}w))(r({\mathcal K}v)u)-r({\mathcal K}v\cdot{\mathcal K}w)\beta(u)\\
    =&0.
\end{align*}
Similarly, one can check the other axioms of Definition \ref{defdiass}.
Thus $(V,\vdash_{\mathcal K},\dashv_{\mathcal K},\beta)$ is a Hom-associative dialgebra.
\end{proof}

For the case of Lie algebras, a relative averaging operator is also named an embedding tensor. In this paper, we use the name "relative averaging operator" to unify nomination for all algebraic structures treated here. The notion of relative averaging operator (an embedding tensor) on a Lie algebra was introduced in \cite{sheng-embedd}. A generalization of this notion on Hom-setting was introduced in \cite{dasmakhlouf}.
In what follow we recall some results. Let $(A,V)$ be a \textsf{HLie-Rep} pair.
\begin{defi}
       A linear map $\mathcal K: V \to A$ is called a relative averaging operator on \textsf{HLie-Rep} pair if $\mathcal K \b =\a \mathcal K$ and
      \begin{align}
  & [\mathcal K(u), \mathcal K(v)]=\mathcal K(\rho(\mathcal K(u))v),\ \forall u,v \in V.
      \end{align}
\end{defi}
Let $(A,V)$ be a \textsf{HLie-Rep} pair.  On the direct sum vector space $A\oplus V$,  define the two maps $\alpha_{A\oplus V}$ and $\{\cdot,\cdot\}_{A\oplus V}$    by
\begin{eqnarray}
    \alpha_{A\oplus V}(x+u)&=&\alpha(x)+\beta(u),\\\{x+u,y+v\}_{A\oplus V}&=&[x,y]+\rho(x)v.
\end{eqnarray}

\begin{pro}\label{hsdirect Lie}
  With the above notations, $(A\oplus V,\{\cdot,\cdot\}_{A\oplus V},\ \alpha_{A\oplus V})$ is Hom-Leibniz algebra, called \textbf{hemisemi-direct product} Hom-Leibniz algebra and denoted by $A\oplus_{Leib} V$.
\end{pro}
 \begin{thm}\label{graphHLie}
Let $(A,V)$ be a \textsf{HLie-Rep} pair. A linear map $\mathcal{K}:V\to A$ is a relative averaging operator on the Hom-Lie algebra $(A,[\cdot,\cdot],\alpha)$ if and only if  the graph $Gr(\mathcal{K}) = \{\mathcal{K}u+u\ |\ u \in V\}$ is a Hom-Leibniz subalgebra of the hemisemi-direct product $A\oplus_{Leib}V$.
\end{thm}
\begin{pro}
   Let $(A,V)$ be a \textsf{HLie-Rep} pair. A linear map $\mathcal K: V \to A$ is a relative averaging operator on $V$ over $A$ if and only if the map
$$N_{\mathcal K}: (A \oplus V) \to (A \oplus V),\ x+u \mapsto \mathcal K(u)$$
is a Nijenhuis operator on the  hemisemi-direct product $A \oplus_{Leib} V$.
\end{pro}
 Similarly as for Hom-associative algebras (Proposition \ref{rel av to av}), a relative averaging operator induces an averaging operator. We will not write the result to avoid repetition.
\begin{thm}\label{LietoLeibniz}

 Let $\mathcal{K}:V\rightarrow A$ be a relative averaging operator  on a \textsf{HLie-Rep} pair algebra $ (A,[\cdot,\cdot],\alpha)$ with respect to the representation  $(V,\rho,\beta )$,   then $(V,\{\cdot,\cdot\}_\mathcal{K},\beta )$ is a   Hom-Leibniz algebra, where
    \begin{eqnarray*}
\{u,v\}_\mathcal{K} &=&\rho(\mathcal{K}(u))(v),\quad \forall u,v\in V.
    \end{eqnarray*}
\end{thm}
We have the following commuting diagram
\begin{equation*}
\xymatrix{\text{Hom-Leib}\ar@{<-}[d]_{\text{Thm}\ \ref{LietoLeibniz}} & &\text{Hom-diass} \ar@{->}_{\text{Prop}\ \ref{asstoLeibniz}}[ll]\ar@{<-}_{\text{Thm}\ \ref{asstodiass}}[d] \\
\text{Hom-Lie} & & \text{Hom-ass} \ar@{->}_{-}[ll] }
\end{equation*}


\subsection{The case of Hom-Jordan}

In this section, we introduce a generalization of Jordan dialgebra by twisting their identities.
We show that this
generalization fits with Hom-associative algebras.
We begin by introducing  Hom-Jordan dialgebras.
\begin{defi}
A (left) Hom-Jordan dialgebra is a vector space $A$ equipped with a binary operation $\bullet$ and a linear map $\a:A \to A$ satisfying
{\small
\begin{align}
&\hspace{4cm}(x_1\bullet x_2)\bullet x_3 = (x_2\bullet x_1)\bullet x_3, \label{HJL0}\\
 \nonumber &((x_4\bullet  x_3)\bullet \alpha(x_2))\bullet \alpha^2(x_1)
+ \alpha^2(x_4)\bullet (\alpha(x_2)\bullet (x_3\bullet  x_1))
+ \alpha^2(x_3)\bullet (\alpha(x_2)\bullet (x_4\bullet  x_1))   \\
 =& \alpha(x_4\bullet  x_3)\bullet(\alpha(x_2)\bullet \alpha( x_1))
 + (\alpha(x_4)\bullet  \alpha(x_2))\bullet\alpha(x_3\bullet  x_1)
 +(\alpha (x_3)\bullet  \alpha(x_2))\bullet\alpha(x_4\bullet  x_1) ,\label{HJL1}\\\nonumber
&
  \alpha^2(x_1)\bullet ((x_4\bullet x_3)\bullet  x_2)
+ \alpha^2(x_4)\bullet ((x_3\bullet  x_1)\bullet  x_2)
+ \alpha^2(x_3)\bullet ((x_4\bullet x_1)\bullet  x_2)   \\
 = &\alpha(x_4\bullet  x_3)\bullet (\alpha(x_1)\bullet  x_2)
 + \alpha(x_1\bullet x_3)\bullet(\alpha (x_4 )\bullet x_2)
 + \alpha(x_4\bullet x_1)\bullet(\alpha(x_3)\bullet x_2).\label{HJL2}
\end{align}}
\end{defi}
When the twisting map $\alpha$ is the identity map, we recover the classical notion of Jordan dialgebra.
\begin{rmk}
Note that identities \eqref{HJL1} and \eqref{HJL2} are equivalent to
\begin{align}
   (x^2,\a(y),\a(z))_\a=& 2(\a(x),\a(y),x\bullet z)_\a,\\
    \a^2(x)\bullet (x^2 \bullet y)=& \a(x^2) \bullet (\a(x) \bullet y),
\end{align}
for any $x,y,z \in J$, where $x^2=x \bullet x$ and $(x,y,z)_\a=(x \bullet y)\bullet \a(z)-\a(x) \bullet (y \bullet z)$.
\end{rmk}

Let $(A,V)$ be a \textsf{HJor-Rep} pair.
\begin{defi}
       A linear map $\mathcal K: V \to A$ is called a relative averaging operator on $A$ with respect to $(V,\pi,\b)$ if $\mathcal K \b =\a \mathcal K$ and
      \begin{align}
  & \mathcal K(u)\o \mathcal K(v)=\mathcal K(\pi(\mathcal K(u))v),\ \forall u,v \in V.
      \end{align}
\end{defi}
On the direct sum vector space $A\oplus V$,  define the two maps $\alpha_{A\oplus V}$ and $\bullet_{A\oplus V}$    by
\begin{eqnarray}
    \alpha_{A\oplus V}(x+u)&=&\alpha(x)+\beta(u),\\(x+u)\bullet_{A\oplus V}(y+v)&=&x\o y+\r(x)v.
\end{eqnarray}

\begin{pro}
  With the above notations, $(A\oplus V,\bullet_{A\oplus V},\ \alpha_{A\oplus V})$ is Hom-Jordan dialgebra. 
\end{pro}
This Hom-Jordan dialgebra is called the  hemisemi-direct product Hom-Jordan dialgebra and denoted by $A\oplus_{DiJor} V$.
\begin{proof}
 For any $x_1,x_2,x_3,x_4\in A$ and $u_1,u_2,u_3,u_4\in V$, we have
{ \small\begin{align*}
&\big(((x_4+u_4)\bullet_{A\oplus V}(x_3+u_3))\bullet_{A\oplus V}\alpha_{A\oplus V}(x_2+u_2)\big)\bullet_{A\oplus V}\alpha^{2}_{A\oplus V}(x_1+u_1)\\&+\alpha^2_{A\oplus V}(x_4+u_4)\bullet_{A\oplus V}\big(\alpha_{A\oplus V}(x_2+u_2)\bullet_{A\oplus V}((x_3+u_3)\bullet_{A\oplus V}(x_1+u_1))\big)\\&+\alpha^2_{A\oplus V}(x_3+u_3)\bullet_{A\oplus V}\big(\alpha_{A\oplus V}(x_2+u_2)\bullet_{A\oplus V}((x_4+u_4)\bullet_{A\oplus V}(x_1+u_1))\big)\\=&\big((x_4\circ x_3)\circ \a(x_2)+\pi(x_4\circ x_3)\b(u_2)\big)\bullet_{A\oplus V}(\a^2(x_1)+\b^2(u_1))\\&+(\a^2(x_4)+\b^2(u_4))\bullet_{A\oplus V}\big(\a(x_2)\circ(x_3\circ x_1)+\pi(\a(x_2))\pi(x_3)u_1\big)\\&+(\a^2(x_3)+\b^2(u_3))\bullet_{A\oplus V}\big(\a(x_2)\circ(x_4\circ x_1)+\pi(\a(x_2))\pi(x_4)u_1\big)\\=&\big((x_4\circ x_3)\circ\a(x_2)\big)\a^2(x_1)+\pi((x_4\circ x_3)\circ \a(x_2))\b^2(u_1)\\&+\a^2(x_4)\circ\big(\a(x_2)\circ(x_3\circ x_1)\big)+\pi(\a^2(x_4))(\pi(\a(x_2))\pi(x_3)u_1)\\&+\a^2(x_3)\circ\big(\a(x_2)\circ(x_4\circ x_1)\big)+\pi(\a^2(x_3))(\pi(\a(x_2))\pi(x_4)u_1).
 \end{align*}}
On the other hand, we have
{\small\begin{align*}
    &\a_{A\oplus V}((x_4+u_4)\bullet_{A\oplus V}(x_3+u_3))\bullet_{A\oplus V}(\a_{A\oplus V}(x_2+u_2)\bullet_{A\oplus V}\a_{A\oplus V}(x_1+u_1))\\&+(\a_{A\oplus V}(x_4+u_4)\bullet_{A\oplus V}\a_{A\oplus V}(x_2+u_2))\bullet_{A\oplus V}\a_{A\oplus V}((x_3+u_3)\bullet_{A\oplus V}(x_1+u_1))\\&+(\a_{A\oplus V}(x_3+u_3)\bullet_{A\oplus V}\a_{A\oplus V}(x_2+u_2))\bullet_{A\oplus V}\a_{A\oplus V}((x_4+u_4)\bullet_{A\oplus V}(x_1+u_1))\\=&\a(x_4\circ x_3)\circ(\a(x_2)\circ \a(x_1))+\pi(\a(x_4\circ x_3))(\pi(\a(x_2))\b(u_1))\\&+(\a(x_4)\circ\a(x_2))\circ\a(x_3\circ x_1)+\pi(\a(x_4)\circ\a(x_2))(\pi(\a(x_3))\b(u_1))\\&+(\a(x_3)\circ\a(x_2))\circ\a(x_4\circ x_1)+\pi(\a(x_3)\circ\a(x_2))(\pi(\a(x_4))\b(u_1)).
\end{align*}}
Then, according to Eq. \eqref{RepHomJor2}, \eqref{HJL1} is satisfied. Analogously, one can check \eqref{HJL0} and \eqref{HJL2}.
Thus, $A\oplus_{DiJor} V$ is a Hom-Jordan dialgebra.
\end{proof}
 \begin{thm}
Let $(A,V)$ be a \textsf{HJor-Rep} pair. A linear map $\mathcal{K}:V\to A$ is a relative averaging operator on the Hom-Jordan algebra $(A,\o,\alpha)$ if and only if  the graph $Gr(\mathcal{K}) = \{\mathcal{K}u+u\ |\ u \in V\}$ is a Hom-Jordan subdialgebra of the hemisemi-direct product $A\oplus_{DiJor}V$.
\end{thm}
The proof is straightforward.
\begin{pro}
    Let $(A,V)$ be a \textsf{HJor-Rep} pair. A linear map $\mathcal K: V \to A$ is a relative averaging operator on $V$ over $A$ if and only if the map
$$N_{\mathcal K}: (A \oplus V) \to (A \oplus V),\ x+u \mapsto \mathcal K(u)$$
is a Nijenhuis operator on the  hemisemi-direct product $A \oplus_{DiJor} V$.
\end{pro}
\begin{thm}\label{JortodiJor}
Let $\mathcal{K}:V\to A$ be a relative averaging operator on the Hom-Jordan algebra $(A,\o, \alpha )$
with respect to the module $(V,\pi,\b)$ .  Then $(V,\bullet, \beta )$
 is a Hom-Jordan dialgebra, where 
$$u\bullet_\mathcal{K} v =\pi(\mathcal{K}(u))v,\quad \forall \ u,v\in V.$$\end{thm}
\begin{proof}
   Let $u_1,u_2,u_3\in V$, since $"\o"$ is commutative then we have
   \begin{align*}
       (u_1\bullet_\mathcal{K} u_2)\bullet_\mathcal{K} u_3-(u_2\bullet_\mathcal{K} u_1)\bullet_\mathcal{K} u_3&=\r(\mathcal{K}(\r(\mathcal{K}(u_1))(u_2)))(u_3)-\r(\mathcal{K}(\r(\mathcal{K}(u_2))(u_1)))(u_3)\\&=\r\big(\mathcal{K}(u_1)\o\mathcal{K}(u_2)-\mathcal{K}(u_2)\o\mathcal{K}(u_1))\big)(u_3)=0.
   \end{align*}

   Therefore  \eqref{HJL0} is satisfied. Now,
   let $u_1,u_2,u_3,u_4\in V$
   \begin{align*}
       &((u_4\bullet_\mathcal{K}  u_3)\bullet_\mathcal{K} \beta(u_2))\bullet_\mathcal{K} \beta^2(u_1)
+ \beta^2(u_4)\bullet_\mathcal{K} (\beta(u_2)\bullet_\mathcal{K} (u_3\bullet_\mathcal{K}  u_1))
+ \beta^2(u_3)\bullet_\mathcal{K} (\beta(u_2)\bullet_\mathcal{K} (u_4\bullet_\mathcal{K}  u_1))   \\
 =&\r\big(\mathcal{K}(\r(\mathcal{K}(\r(\mathcal{K}(u_4))(u_3)))(\beta(u_2)))\big)(\beta^2(u_1))+\r(\mathcal{K}(\beta^2(u_4)))\big(\r(\mathcal{K}(\beta(u_2)))\big)\big(\r(\mathcal{K}(u_3))(u_1)\big) \\
&+\r(\mathcal{K}(\beta^2(u_3)))\big(\r(\mathcal{K}(\beta(u_2)))\big)\big(\r(\mathcal{K}(u_4))(u_1)\big)\\=&\r\big(\mathcal{K}(\r(\mathcal{K}u_4)(u_3))\o\alpha(\mathcal{K}u_2)\big)(\beta^2(u_1))+\r(\alpha^2(\mathcal{K}u_4))\big(\r(\alpha(\mathcal{K}u_2))\big)\big(\r(\mathcal{K}u_3)(u_1)\big) \\
&+\r(\alpha^2(\mathcal{K}u_3))\big(\r(\alpha(\mathcal{K}u_2))\big)\big(\r(\mathcal{K}u_4)(u_1)\big) \\=&\r\big((\mathcal{K}u_4\o\mathcal{K}u_3)\o\alpha(\mathcal{K}u_2)\big)(\beta^2(u_1))+\r(\alpha^2(\mathcal{K}u_4))\big(\r(\alpha(\mathcal{K}u_2))\big)\big(\r(\mathcal{K}u_3)(u_1)\big) \\
&+\r(\alpha^2(\mathcal{K}u_3))\big(\r(\alpha(\mathcal{K}u_2))\big)\big(\r(\mathcal{K}u_4)(u_1)\big) .
   \end{align*}
   On the other hand we have
   \begin{align*}
       &\beta(u_4\bullet_\mathcal{K} u_3)\bullet_\mathcal{K}(\beta(u_2)\bullet_\mathcal{K}\beta(u_1))+(\beta(u_4)\bullet_\mathcal{K}\beta(u_2))\bullet_\mathcal{K}\beta(u_3\bullet_\mathcal{K}u_1)+(\beta(u_3)\bullet_\mathcal{K}\beta(u_2))\bullet_\mathcal{K}\beta(u_4\bullet_\mathcal{K}u_1)\\=&\r\big(\mathcal{K}(\beta(\r(\mathcal{K}u_4)(u_3)))\big)\big(\r(\mathcal{K}(\beta(u_2)))\beta(u_1)\big)
      +\r\big(\mathcal{K}(\r(\mathcal{K}(\beta(u_4)))(\beta(u_2))\big)\big(\beta(\r(\mathcal{K}u_3(u_1)))\big)\\&+\r\big(\mathcal{K}(\r(\mathcal{K}(\beta(u_3)))(\beta(u_2))\big)\big(\beta(\r(\mathcal{K}u_4(u_1)))\big)\\=&\r\big(\alpha(\mathcal{K}(\r(\mathcal{K}u_4)(u_3)))\big)\big(\r(\alpha(\mathcal{K}u_2))(\beta(u_1))\big)+\r\big(\mathcal{K}(\r(\mathcal{K}(\beta(u_4)))(\beta(u_2))\big)\big(\r(\alpha(\mathcal{K}u_3))(u_1)\big)\\&+\r\big(\mathcal{K}(\r(\mathcal{K}(\beta(u_3)))(\beta(u_2))\big)\big(\r(\alpha(\mathcal{K}u_4))(u_1)\big)\\=&\r\big(\alpha(\mathcal{K}u_4)\o\alpha(\mathcal{K}u_3)\big)\big(\r(\alpha(\mathcal{K}u_2))(\beta(u_1))\big)+\r\big(\mathcal{K}(\beta(u_4))\o\mathcal{K}(\beta(u_2))\big)\big(\r(\alpha(\mathcal{K}u_3))(u_1)\big)\\&+\r\big(\mathcal{K}(\beta(u_3))\o\mathcal{K}(\beta(u_2))\big)\big(\r(\alpha(\mathcal{K}u_4))(u_1)\big)\\=&\r\big(\alpha(\mathcal{K}u_4)\o\alpha(\mathcal{K}u_3)\big)\big(\r(\alpha(\mathcal{K}u_2))(\beta(u_1))\big)+\r\big(\alpha(\mathcal{K}(u_4))\o\alpha(\mathcal{K}(u_2))\big)\big(\r(\alpha(\mathcal{K}u_3))(\beta(u_1))\big)\\&+\r\big(\alpha(\mathcal{K}(u_3))\o\alpha(\mathcal{K}(u_2))\big)\big(\r(\alpha(\mathcal{K}u_4))(\beta(u_1))\big).
   \end{align*}
   Thus, according to Eq. \eqref{RepHomJor2}, we get
   \begin{multline*}
  ((u_4\bullet_\mathcal{K}  u_3)\bullet_\mathcal{K} \beta(u_2))\bullet_\mathcal{K} \beta^2(u_1)
+ \beta^2(u_4)\bullet_\mathcal{K} (\beta(u_2)\bullet_\mathcal{K} (u_3\bullet_\mathcal{K}  u_1))
+ \beta^2(u_3)\bullet_\mathcal{K} (\beta(u_2)\bullet_\mathcal{K} (u_4\bullet_\mathcal{K}  u_1))     \\
 = \beta(u_4\bullet_\mathcal{K} u_3)\bullet_\mathcal{K}(\beta(u_2)\bullet_\mathcal{K}\beta(u_1))+(\beta(u_4)\bullet_\mathcal{K}\beta(u_2))\bullet_\mathcal{K}\beta(u_3\bullet_\mathcal{K}u_1)+(\beta(u_3)\bullet_\mathcal{K}\beta(u_2))\bullet_\mathcal{K}\beta(u_4\bullet_\mathcal{K}u_1) .
\end{multline*}
Similarly, by Eq. \eqref{RepHomJor1}, we can check the identity \eqref{HJL1}.
\end{proof}

\begin{defi}
    Let $(A_1,\bullet_1,\alpha_1)$ and $(A_2,\bullet_2,\alpha_2)$ be two Hom-Jordan dialgebras. A linear map  $\phi:A_1\to A_2$ is called a morphism of Hom-Jordan dialgebras if it's satisfying, for all $x,y\in A_1$,

       $$\phi\alpha_1=\alpha_2 \phi\quad
       \text{and}\quad \phi(x\bullet_1y)=\phi(x)\bullet_2 \phi(y).$$

\end{defi}
The following theorem gives a procedure to construct a Hom-Jordan dialgebras using Hom-Jordan dialgebras and their algebra endomorphisms.
\begin{thm}
    Let $(A,\bullet,\alpha)$ be a Hom-Jordan dialgebra and $\phi:A\to A$ be a morphism then $$A_\phi=(A,\bullet_\phi=\phi\bullet,\phi\alpha)$$ is also a Hom-Jordan dialgebra.
\end{thm}
\begin{cor}
     Let $(A,\bullet)$ be a Jordan dialgebra and $\alpha:A\rightarrow A$ be an algebra endomorphism.  Then
$(A,\bullet_\alpha,\alpha)$, where $\bullet_\alpha=\alpha\bullet$, is a Hom-Jordan dialgebra.

\end{cor}
\emptycomment{\begin{ex}
    Let $\{e_1,e_2,e_3\}$ be a basis of   a $3$-dimensional vector space $A$ over $\mathbb{K}$. The following product $\mu $ and the linear map $\alpha$ define Hom-Jordan algebras over $\mathbb{K}$.
$$
\begin{array}{c|c|c|c}
  \mu  & e_1 & e_2 & e_3 \\
  \hline
  e_1& ae_1 & ae_2 & be_3 \\
  \hline
  e_2 & ae_2 & ae_2 & \frac{b}{2}e_3 \\
  \hline
  e_3 & be_3 & \frac{b}{2}e_3 & 0
\end{array},$$
\begin{align*}
    \alpha(e_1)=ae_1,\ \alpha(e_2)=ae_2,\ \alpha(e_3)=be_3,
\end{align*}
where $a$ and $b$ are parameters in $\mathbb{K}$.  Let $A$ be the operator defined with respect to the basis $\{e_1,e_2,e_3\}$ by
\begin{align*}
    A(e_1)=\lambda_1 e_1,\ A(e_2)=\lambda_2 e_2,\  A(e_3)=\lambda_3 e_3,
\end{align*}
where $\lambda_1,\lambda_2$ and $\lambda_3$ are parameters in $\mathbb{K}$.  Then we can easily check that $A$ is an averaging operator on $A$.
Now, using  the previous corollary,   there is a Hom-Jordan dialgebra structure
on $A$, with the same twist map and the multiplication given by
$x\cdot y= \mu(A(x),y), \quad \forall x,y \in A$, that is

$$
\begin{array}{c|c|c|c}
  \cdot  & e_1 & e_2 & e_3 \\
  \hline
  e_1& \lambda_1 a e_1 & \lambda_1 a e_2 & \lambda_1 b e_3 \\
  \hline
  e_2 & \lambda_2 a e_3 & \lambda_2 ae_2 & \lambda_2 \frac{b}{2}e_3\\
  \hline
  e_3 & \lambda_3 b e_3 & \lambda_3 \frac{b}{2}e_3  & 0
\end{array}.$$
\end{ex}}
\begin{thm}\label{diasstodijor}
 Let $\mathcal{K}:V\to A$ be a relative averaging operator on a Hom-associative algebra $(A,\c,\alpha)$ with respect to an $A$-bimodule $(V,l,r,\b)$. Then $\mathcal{K}$ is a relative averaging operator on the associated Hom-Jordan algebra  $(A,\o=\c+\c^{op},\alpha)$, where $x\c^{op} y=y\c x$, with respect to an $A$-module $(V,\pi=l+r,\b)$.

 Moreover, if $(V,\lprod,\rprod,\b)$ is the Hom-associative dialgebra associated to the Hom-associative algebra $(A,\c,\alpha)$. Then the anti-dicommutator defined in Eq. \eqref{antidicomm} makes $V$ into a Hom-Jordan dialgebra.
\end{thm}
\begin{proof}
    For any $u,v\in V$, we have
    \begin{eqnarray*}
        \mathcal{K}(u)\o \mathcal{K}(v)&=&\mathcal{K}(u)\c \mathcal{K}(v)+\mathcal{K}(v)\c \mathcal{K}(u)\\&=&\mathcal{K}(l(\mathcal{K}(u))v)+\mathcal{K}(r(\mathcal{K}(u))v)\\&=&\mathcal{K}(l(\mathcal{K}(u))v+r(\mathcal{K}(u))v)\\&=&\mathcal{K}(\pi(\mathcal{K}(u))v).
    \end{eqnarray*}
Then $\mathcal{K}$ is a relative averaging operator on the Hom-Jordan algebra $(A,\o,\a)$.

Furthermore, it follows from Theorem \ref{asstodiass} that $(V,\lprod,\rprod,\b)$ is a Hom-associative dialgebra where $$u\lprod v=l(\mathcal{K}(u))v,\quad u\rprod v=r(\mathcal{K}(v))u, $$
on the other hand, according to Theorem \ref{JortodiJor},  we have $(V,\bullet,\b)$ is a Hom-Jordan dialgebra where
$$u\bullet v=\pi(\mathcal{K}(u)v)=l(\mathcal{K}(u))v+r(\mathcal{K}(u))v=u\lprod v+v\rprod u.$$
This completes the
proof.
\end{proof}
Thus, we have the following diagram
\emptycomment{\begin{equation*}
\xymatrix{\text{Hom-diass}\ar@{<-}[d]_{Thm\ \ref{asstodiass}} & &\text{Hom-diJor} \ar@{<-}_{Thm\ \ref{diasstodijor}}[ll]\ar@{<-}_{Thm\ \ref{JortodiJor}}[d] \\
\text{Hom-ass} & & \text{Hom-Jordan} \ar@{<-}_{-}[ll] }
\end{equation*}}

\begin{equation*}
\xymatrix{\text{Hom-Leib}\ar@{<-}[d]_{\text{Thm}\ \ref{LietoLeibniz}} & &\ar@{->}_{\text{Prop}\ \ref{asstoLeibniz}}[ll] \text{Hom-diass}\ar@{<-}[d]_{\text{Thm}\ \ref{asstodiass}} & &\text{Hom-diJor} \ar@{<-}_{\text{Thm}\ \ref{diasstodijor}}[ll]\ar@{<-}_{\text{Thm}\ \ref{JortodiJor}}[d] \\
\text{Hom-Lie} & & \text{Hom-ass} \ar@{->}_{-}[ll] & & \text{Hom-Jordan} \ar@{<-}_{+}[ll] }
\end{equation*}

\section{Triplications via homomorphic relative averaging operator}\label{triplication}

In this section, we introduce the notion of a homomorphic relative averaging operator  on a Hom-algebra which allows 
the replicating of  the multiplication into three multiplications of the same nature. But notice that homomorphic relative averaging operators do not generalize averaging operators as in the case of relative  Rota-Baxter operators. On the other hand, in \cite{tangsheng}, the authors introduce a weighted version of relative averaging operator called non-abelian embedding tensor. These operators give rise to other algebraic structures such  as  Lie-Leibniz algebras from Lie algebras. A homomorphic relative  averaging operator   on a Lie algebra  induces a triLeibniz algebra.
\subsection{Homomorphic relative averaging operator}
\begin{defi}
Let $(A,V)$ be a \textsf{HAss-Act} pair. A linear map $\mathcal{H}: V \to A$  is called a homomorphic relative  averaging operator,
     if it is a relative averaging operator and for any $u,v \in V$,
     \begin{align}
        \label{avop2} \mathcal{H}(u)\cdot\mathcal{H}(v)=& \mathcal{H}(u\cdot_V v).
     \end{align}
\end{defi}
\begin{ex}
    Let $(A,\c,\a)$ be a Hom-associative algebra. Note that the space $B = \underbrace{A\oplus\cdots\oplus A}_{n \text{ summand}}$ can be given a Hom-associative algebra structure with the multiplication given by the componentwise multiplication of $A$, and the twisting map $\b$ given by
    $$\b(a_1,\ldots,a_n)=(\a(a_1),\ldots,\a(a_n))$$. The space $B$ is a Hom-associative $A$-bimodule with the left and right $A$-actions on $B$ are respectively given by
    \begin{align}\label{left-right-b}
  l(a)(a_1,\ldots,a_n) = (l(a) a_1,\ldots,l(a) a_n)~~ \text{ and } ~~  r(a)(a_1,\ldots,a_n) = (r(a)a_1 ,\ldots,r(a)a_n ),
\end{align}
for $a\in A$  and $(a_1,\ldots,a_n)\in B$. Then for any $1\leq i\leq n$, the $i$-th projection map $P_i : B \rightarrow A$ is a homomorphic relative averaging operator.
\end{ex}
\begin{ex} A crossed module of Hom-associative algebras is a quadruple $(A, V, d,(l,r))$ in which $(A,\c,\a)$ is a Hom-associative algebra, $(V,l,r,\b,\c_V)$ is an action and $d:V\to A$ is a morphism of Hom-associative algebra, where
\begin{align*}
d (l(x)u) = x \cdot d(u), \quad d(r(x)u)= d(u) \cdot x, \quad l(d(u))v = r(d(v))u = u \c_V v,
\end{align*} for all $  x \in A, u, v \in V.$
Thus, $d$ is a homomorphic  relative averaging operator of the Hom-associative algebra $(A,\c,\a)$.
\end{ex}

Recall that, an $\mathcal{O}$-operator of weight $\lambda$  is a linear map  $O: V \to A$ satisfying $O\b=\a O$, and for any $u,v\in V$,
\begin{equation}
    \label{O-Operator}Ou\cdot Ov=O(l(Ou)v+r(Ov)u+\lambda(u\cdot_Vv)).
\end{equation}
\begin{lem}\label{ooperator}
Let $\mathcal{H}$ be a homomorphic relative averaging operator on a \textsf{HAss-Act pair} $(A,V)$. Then the linear operator $O=-\mathcal{H}$ is an $\mathcal{O}$-operator of weight $-1$ on $A$.
\end{lem}

\begin{proof}
Let $u,v\in V$, we have
\begin{align*}
    &(-\mathcal{H}(u))\cdot(-\mathcal{H}(v))-(-\mathcal{H}(l(-\mathcal{H}(u))v))+(-\mathcal{H}(r(-\mathcal{H}(v))u))-((-1)(-\mathcal{H}(u\cdot_Vv)))\\=&\mathcal{H}(u)\cdot\mathcal{H}(v)-\mathcal{H}(l(\mathcal{H}u)v)+\mathcal{H}(r(\mathcal{H}v)u)- \mathcal{H}(u\cdot_Vv) =0.  \end{align*}
    Thus, $-\mathcal{H}$ is an $\mathcal{O}$-operator of weight $-1$ on $A$.
\end{proof}
\begin{defi}\label{triass}
A \emph{Hom-associative trialgebra} (also called Hom-associative trialgebra) is a vector space $A$ that comes equipped with three binary operations, $\rprod$ (left), $\lprod$ (right), $\mprod$ (middle), and a linear map $\alpha:A\rightarrow A$, where $(A,\lprod,\rprod,\a)$ is a Hom-associative dialgebra, $(A,\mprod,\a)$ is a Hom-associative algebra and for all $x, y, z \in A$:
\begin{align}
   (x \rprod y) \rprod \alpha(z) & = \alpha(x) \rprod (y \,\mprod\, z), \label{axiom6} \\
   (x \,\mprod\, y) \rprod \alpha(z) & = \alpha(x) \,\mprod\, (y \rprod z), \label{axiom7} \\
   (x \rprod y) \,\mprod\, \alpha(z) & = \alpha(x) \,\mprod\, (y \lprod z), \label{axiom8} \\
   (x \lprod y) \,\mprod\, \alpha(z) & = \alpha(x) \lprod (y \,\mprod\, z), \label{axiom9} \\
   (x \,\mprod\, y) \lprod \alpha(z) & = \alpha(x) \lprod (y \lprod z), \label{axiom10}
   \end{align}

When the twisting map $\alpha$ is the identity map, we recover the classical
notion of associative trialgebras.
\end{defi}
  Let $(A,\lprod,\rprod,\mprod,\alpha)$ and $(A^\prime,\lprod^\prime,\rprod^\prime,\mprod^\prime,\alpha^\prime)$ be a two Hom-triassociative  algebras, a morphism of Hom-associative trialgebras is a morphism of Hom-associative dialgebra satisfying, for all $x,y\in A$
    $$\phi(x\mprod y)=\phi(x)\mprod^\prime \phi(y).
$$
\begin{rmk}
    If $(A,\lprod,\rprod,\alpha)$ is a Hom-associative dialgebra, then $(A,\lprod,\rprod,0,\alpha)$ is a Hom-associative trialgebra.
\end{rmk}
\begin{thm}
Let $(A,\lprod,\rprod,\mprod,\alpha)$ be a Hom-associative trialgebra and $\phi:A \to A$ be a morphism, then
$$A_\phi=(A,\lprod_\phi=\phi\lprod,\rprod_\phi=\phi\rprod,\mprod_\phi=\phi\mprod,\phi\alpha)$$
is also a Hom-associative trialgebra. Moreover, if $A$ is multiplicative, then
so is $A_\phi$.
\end{thm}
\begin{cor}\label{twisttriass}
    Let $(A,\lprod,\rprod,\mprod)$ be an associative trialgebra and $\alpha:A\to A$ be an algebra endomorphism. Then $(A,\lprod_\alpha,\rprod_\alpha,\mprod_\alpha,\alpha)$  is a Hom-associative trialgebra, where $$\lprod_\alpha=\alpha\lprod,\quad \rprod_\alpha=\alpha\rprod \quad \text{and}\quad\mprod_\alpha=\alpha\mprod.$$
\end{cor}
Now we give an example.
\begin{ex}\label{extriass}
Let $\{e_1,e_2\}$ be a basis of   a $2$-dimensional vector space $A$ over $\mathbb{K}$. In \cite{Mainellis} the author shows that the following products $\lprod,\rprod$ and $\mprod $ define an associative trialgebra over $\mathbb{K}$.
$$e_1\lprod e_1=e_1\mprod e_1=e_1\mprod e_1=e_1,$$
    $$
    e_1\lprod e_2=e_2\rprod e_1=e_1\mprod e_2=e_2.$$

    We construct a twisting map $\alpha$ on $A$ as follows
    \begin{align*}
    \alpha(e_1)=ae_1,\ \alpha(e_2)=be_2,
\end{align*}
where $a,b$ and $c$ are parameters in $\mathbb{K}$.
Then according to Corollary \ref{twisttriass}, $(A,\lprod_\alpha,\rprod_\alpha,\mprod_\alpha,\alpha)$ is a Hom-associative trialgebra, where
$$e_1\lprod_\alpha e_1=e_1\mprod e_1=e_1\mprod_\alpha e_1=ae_1,$$
    $$
    e_1\lprod_\alpha e_2=e_2\rprod_\alpha e_1=e_1\mprod_\alpha e_2=be_2.$$
\end{ex}
Let $(A,V)$ be a \textsf{HAss-Act} pair. Define on the direct sum $A\oplus V$ the mutiplications $\lprod_{A\oplus V},\rprod_{A\oplus V}$ and $\mprod_{A\oplus V}$, and the map $\alpha_{A\oplus V}$ by
\begin{eqnarray}
    (x+u)\lprod_{A\oplus V} (y+v)&=&x\cdot y+l(x)v,\\
    (x+u)\rprod_{A\oplus V} (y+v)&=&x\cdot y+r(y)u,\\(x+u)\mprod_{A\oplus V} (y+v)&=&x\cdot y+u\cdot_Vv,\\
     \alpha_{A\oplus V}(x+u)&=&\alpha(x)+\beta(u).
\end{eqnarray}
In the following, we give a characterization of a homomorphic relative averaging operator in terms
of the graph of the operator. We first prove the following result.
\begin{pro}
    With the above notations $(A\oplus V,\lprod_{A\oplus V},\rprod_{A\oplus V},\mprod_{A\oplus V},\alpha_{A\oplus V})$ is a Hom-associative trialgebra, called \textbf{hemisemi-direct product} Hom-associative trialgebra and denoted by $A\oplus_{Triass}V$.

\end{pro}
\begin{proof}
According to Proposition \ref{hsdirect diass}, $(A\oplus V,\lprod_{A\oplus V},\rprod_{A\oplus V},\alpha_{A\oplus V})$ is a Hom-associative dialgebra. It's easy to check that $(A\oplus V,\mprod_{A\oplus V},\a_{A\oplus V})$ is a Hom-associative algebra. Now, for any $x,y,z\in A$ and $u,v,w\in V$, we have
\begin{align*}
    &((x+u)\mprod_{A\oplus V}(y+v))\rprod_{A\oplus V}\a_{A\oplus V}(z+w)-\a_{A\oplus V}(x+u)\mprod_{A\oplus V}((y+v)\rprod_{A\oplus V}(z+w))\\=&(x\c y +u\c_V v)\rprod_{A\oplus V}(\a(z)+\b(w))-(\a(x)+\b(u))\mprod_{A\oplus V}(y\c z+r(z)v)\\=&(x\c y)\c\a(z)+r(\a(z))(u\c_Vv)-\a(x)\c(y\c z)-\b(u)\c_V(r(z)v)\\\overset{\eqref{actass}}{=}&0.
\end{align*}
The proofs of the other axioms of Definition~\ref{triass} are similar.
Hence the result follows.
\end{proof}
We have the following theorem whose proof is similar to the proof of Theorem~\ref{graphHass}.
\begin{thm}
  Let $(A,V)$ be a \textsf{HAss-Act} pair. A linear map $\mathcal{H}:V\to A$ is a homomorphic relative averaging operator on the Hom-associative algebra $(A,\cdot,\alpha)$ if and only if  the graph $Gr(\mathcal{H}) = \{\mathcal{H}u+u\ |\ u \in V\}$ is a Hom-triassociative subalgebra of the hemisemi-direct product  Hom-associative trialgebra  $A\oplus_{Triass}V$.
\end{thm}

\begin{pro}
    Let $(A,V)$ be a \textsf{HAss-Act} pair. A linear map $\mathcal H: V \to A$ is a homomorphic relative averaging operator on $V$ over $A$ if and only if the map
$$N_{\mathcal H}: (A \oplus V) \to (A \oplus V),\ x+u \mapsto \mathcal H(u)$$
is a Nijenhuis operator on the  hemisemi-direct product $A \oplus_{Triass} V$.
\end{pro}
\begin{thm}\label{triassfromass}
    Let $\mathcal{H}:V\to A$ be a homomorphic relative averaging operator on the \textsf{HAss-Act} pair $(A,V)$.
     Then the
following operations make $V$ into a Hom-associative trialgebra:
\begin{align*}
    u\rprod_\mathcal{H} v&=r(\mathcal{H}(v))u,\quad u\lprod_\mathcal{H} v=l(\mathcal{H}(u))v,\quad u\mprod_\mathcal{H} v=u\cdot_V v.
\end{align*}
\end{thm}
\begin{proof}
    According to Theorem \ref{asstodiass}, $(V,\lprod_\mathcal{H},\rprod_\mathcal{H},\b)$ is a Hom-associative dialgebra.
 It's obvious that $(V,\mprod_\mathcal{H},\b)$ is a Hom-associative algebra.
    For any $u,v,w\in V$
    \begin{eqnarray*}
        (u\rprod_{\mathcal{H}} v)\rprod_{\mathcal{H}} \b(w)-\b(u)\rprod_{\mathcal{H}}(v\mprod_{\mathcal{H}} w)&=&(r(\mathcal{H}(v))u)\rprod_{\mathcal{H}}\b(w)-\b(u)\rprod_{\mathcal{H}}(v\cdot_Vw)\\&=&r(\mathcal{H}(\b(w))(r(\mathcal{H}(v))u)-r(\mathcal{H}(v\cdot_Vw))(\b(u))\\&=&r(\a(\mathcal{H}(w))(r(\mathcal{H}(v))u)-r(\mathcal{H}(v\cdot_Vw))(\b(u))\\& \overset{\eqref{avop2}}{=}&0.
    \end{eqnarray*}
    Analogously, one can check the other axioms of Definition \ref{triass}. Thus, $(V,\lprod_{\mathcal{H}},\rprod_{\mathcal{H}},\mprod_{\mathcal{H}},\b)$ is a Hom-associative trialgebra.
\end{proof}

\begin{pro}
    Let $(A,\lprod,\rprod,\mprod,\a)$ be a Hom-associative trialgebra, then the following operations make $A$ into a Hom-tridendriform algebra $$x\prec y=-x\rprod y,\quad x\succ y=-x\lprod y,\quad x\cdot y=x\mprod y.$$
\end{pro}
For the definition of Hom-tridendriform algebra, see \cite{Makhlouf1}.
\emptycomment{\begin{proof}
    For any $x,y,z\in A$ we have
    \begin{align*}
        &(x\prec y)\prec \a(z)-\a(x)\prec(y \prec z+y\succ z+y\cdot z)\\=&(x\rprod y)\rprod\a(z)-\a(x)\rprod(y\rprod z+y\lprod z-y\mprod z)\\=&\a(x)\rprod(y\mprod z)-\a(x)\rprod(y\mprod z)
    \end{align*}
\end{proof}
\begin{cor}
    Let $(A,\cdot,\a)$ and $P:A\to A$ a homomorphic averaging operator. Then, $(A,-\rprod,-\lprod,\mprod)$ is Hom-tridendriform algebra, where the operations $\rprod,\lprod,\mprod$ are defined in Proposition \ref{triassfromass}.
\end{cor}}
\subsection{From Hom-associative trialgebras to Hom-triLeibniz algebras}
In the following, we introduce the notion of homomorphic relative averaging operator on Hom-Lie algebra with a given action.
\begin{defi}
       A linear map $\mathcal H: V \to A$ is called a homomorphic relative averaging operator on the \textsf{HLie-Act} pair $(A,V)$ if it's a relative averaging operator and a morphism, that is
      \begin{align}
  & \mathcal{H}[u,v]_V=[\mathcal H(u), \mathcal H(v)],\ \forall u,v \in V.
      \end{align}
\end{defi}

In \cite{pei-rep} the authors give the definition of (right) Leibniz trialgebra, next we will introduce the notion of left Hom-Leibniz trialgebra.
\begin{defi}
A left Hom-Leibniz trialgebra is a quadruple $(A,\{\cdot,\cdot\},[\cdot,\cdot],\alpha)$ consisting of a Hom-Lie algebra $(A,[\cdot,\cdot],\a)$ and a Hom-Leibniz algbra $(A,\{\cdot,\cdot\},\a)$ such that, for any $x,y, z \in A$
    \begin{eqnarray}
& \{\a(x),[y,z]\}=[\{x,y\},\a(z)]+[\a(y),\{x,z\}] ,\label{eq:trileib1}\\
& \{[x,y],\a(z)\}=\{\{x,y\},\a(z)\}.& \label{eq:trileib2}
\end{eqnarray}
\end{defi}
\begin{rmk}
    If we have $\{x,y\}^\prime=\{y,x\}$ then $(A,\{\c,\c\}^\prime,[\c,\c],\a)$ is a right Hom-Leibniz trialgebra.
\end{rmk}

On the direct sum vector space $A\oplus V$,  define the two bilinear maps $[\cdot,\cdot]_{A\oplus V}$ and $\{\cdot,\cdot\}_{A\oplus V}$ and the linear map $\alpha_{A\oplus V}$ by
\begin{eqnarray}
[x+u,y+v]_{A\oplus V}&=&[x,y]+[u,v]_V,\\
    \{x+u,y+v\}_{A\oplus V}&=&[x,y]+\rho(x)v,\\
 \alpha_{A\oplus V}(x+u)&=&\alpha(x)+\beta(u).
\end{eqnarray}

\begin{pro}
  With the above notations, $(A\oplus V,\{\cdot,\cdot\}_{A\oplus V},[\cdot,\cdot]_{A\oplus V},\alpha_{A\oplus V})$ is a Hom-Leibniz trialgebra, called hemisemi-direct product Hom-Leibniz trialgebra and denoted by $A\oplus_{TriLeib}V$.

\end{pro}
\begin{proof}
    It follows from Proposition \ref{hsdirect Lie} that $(A\oplus V,\{\c,\c\}_{A\oplus V},\a)$ is a Hom-Leibniz algebra.
On the other hand, it is straightforward to observe that $(A\oplus V,[\c,\c]_{A\oplus V},\a)$ is a Hom-Lie algebra.
Let $x,y,z\in A$ and $u,v,w\in V$, according to Eq. \eqref{actLie}, we have
    \begin{align*}
        &\{\a_{A\oplus V}(x+u),[y+v,z+w]_{A\oplus V}\}_{A\oplus V}-[\{x+u,y+v\}_{A\oplus V},\a_{A\oplus V}(z+w)]_{A\oplus V}\\&-[\a_{A\oplus V}(y+v),\{x+u,z+w\}_{A\oplus V}]_{A\oplus V}\\=&[\a(x),[y,z]]+\rho(\a(x))[v,w]_V-[[x,y],\a(z)]-[\rho(x)v,\b(w)]_V-[\a(y),[x,z]]-[\b(v),\rho(x)w]_V\\=& 0.
    \end{align*}
    Hence, the Eq. \eqref{eq:trileib1} is satisfied. The same for Eq. \eqref{eq:trileib2}. This completes the proof.
\end{proof}
\begin{thm}
Let $(A,V)$ be a \textsf{HLie-Act} pair. A linear map $\mathcal{H}:V\to A$ is a homomorphic relative averaging operator on the Hom-Lie algebra $(A,[\cdot,\cdot],\alpha)$ if and only if  the graph $Gr(\mathcal{H}) = \{\mathcal{H}u+u\ |\ u \in V\}$ is a Hom-Leibniz subtrialgebra of the hemisemi-direct product Hom-Leibniz trialgebra $A\oplus_{TriLeib}V$.
\end{thm}
\begin{proof}
Similar to Theorem \ref{graphHLie}.
\end{proof}
\begin{pro}
    Let $(A,V)$ be a \textsf{HLie-Act} pair. A linear map $\mathcal H: V \to A$ is a homomorphic relative averaging operator on $V$ over $A$ if and only if the map
$$N_{\mathcal H}: (A \oplus V) \to (A \oplus V),\ x+u \mapsto \mathcal H(u)$$
is a Nijenhuis operator on the  hemisemi-direct product $A \oplus_{TriLeib} V$.
\end{pro}
\begin{thm}\label{triasstotriLeib}
    Let $(A,\lprod,\rprod,\mprod,\a)$ be a Hom-associative trialgebra. Define new binary operations
    \begin{align*}
        \{x, y\}&=x\lprod y-y\rprod x,\quad\quad [x,y]=x\mprod y-y\mprod x.
    \end{align*}
    Then, $(A,\{\c,\c\},[\cdot,\cdot],\alpha)$ is a Hom-Leibniz trialgebra.
\end{thm}
\begin{proof}
    Since $(A,\lprod,\rprod,\mprod,\a)$ is a Hom-associative trialgebra, then $(A,\lprod,\rprod,\a)$ is a Hom-associative dialgebra. Thus, according to Remark \ref{asstoLeibniz} $(A,\{\c,\c\},\a)$ is a Hom-Leibniz algebra.

    Since $(A,\mprod,\a)$ is a Hom-associative algebra, then $(A,[\cdot,\cdot],\a)$ is a Hom-Lie algebra.
 Still to verify \eqref{eq:trileib1} and \eqref{eq:trileib2}.

 For any $x,y,z\in A$, we have from Definition \ref{triass} that
 \begin{eqnarray*}
     [\{x,y\},\a(z)]+[\a(y),\{x,z\}]-\{\a(x),[y,z]\}&=&(x\lprod y)\mprod \a(z)-(y\rprod x)\mprod\a(z)-\a(z)\mprod (x\lprod y)\\&+&\a(z)\mprod(y\rprod x)+\a(y)\mprod(x\lprod z)-\a(y)\mprod(z\rprod x)\\&-&(x\lprod z)\mprod\a(y)+(z\rprod x)\mprod \a(y)-\a(x)\lprod(y\mprod z)\\&+&\a(x)\lprod(z\mprod y)+(y\mprod z)\rprod\a(x)-(z\mprod y)\rprod\a(x)\\&=&0.
 \end{eqnarray*}
 On the other hand, we have
 \begin{eqnarray*}
     \{[x,y],\a(z)\}-\{\{x,y\},\a(z)\}&=&(x\mprod y)\lprod\a(z)-(y\mprod x)\lprod\a(z)-\a(z)\rprod(x\mprod y)+\a(z)\rprod(y\mprod x)\\&-&(x\lprod y)\lprod\a(z)+(y\rprod x)\lprod\a(z)+\a(z)\rprod(x\lprod y)-\a(z)\rprod(y\rprod x)\\& \overset{\ref{triass}}{=}&\a(x)\lprod(y\lprod z)-\a(y)\lprod(x\lprod z)-(z\rprod x)\rprod\a(y)+(z\rprod y)\rprod\a(x)\\&-&(x\lprod y)\lprod\a(z)+(y\rprod x)\lprod\a(z)+\a(z)\rprod(x\lprod y)-\a(z)\rprod(y\rprod x)\\& \overset{\ref{defdiass}}{=}&0.
     \end{eqnarray*}
     Thus, the result follows.
\end{proof}
Moreover, we have the following commuting diagram.

\begin{equation*}
\xymatrix{\text{Hom-TriLeib}\ar@{<-}[d]_{(\{\c,\c\},\a)\to(\{\c,\c\},0,\a)} & &\text{Hom-Triass} \ar@{->}_{\text{Thm}\ \ref{triasstotriLeib}}[ll]\ar@{<-}_{(\vdash,\dashv,\a)\to(\vdash,\dashv,0,\a)}[d] \\
\text{Hom-Leib} & & \text{Hom-Diass} \ar@{->}_{\text{Prop}\ \ref{asstoLeibniz} }[ll] }
\end{equation*}
\begin{ex}
It follows from Example \ref{extriass} and Theorem \ref{triasstotriLeib} that,   if $\{e_1,e_2\}$ is a basis of $2$-dimensional vector space $A$ over $\mathbb K$, then the following brackets define a Hom-Leibniz trialgebra structure over $\mathbb K$
$$\{e_1,e_1\}=ae_1,\quad [e_1,e_2]=be_2,$$
where $a,b$ and $c$ are parameters in $\mathbb{K}$.
\end{ex}
\begin{thm}
    Let $\mathcal{H}:V\to A$ be a homomorphic relative averaging operator on the \textsf{HLie-Act} pair $(A,V)$. Then $(V,\{\c,\c\}_\mathcal{H},[\c,\c]_V,\b)$ is a Hom-triLeibniz algebra, where
    $$\{u, v\}_\mathcal{H}=\rho(\mathcal{H}u)v,$$
\end{thm}
\begin{proof}
    We have that $(V,[\c,\c]_V,\b)$ is a Hom-Lie algebra, and it follows from Theorem \ref{LietoLeibniz} that $(V,\{\c,\c\}_\mathcal{H},\b)$ is a Hom-Leibniz algebra. 
    For any $u,v,w\in V$, we have
    \begin{eqnarray*}
        \{\b(u),[v,w]_V\}_\mathcal{H}&=&\rho(\mathcal{H}(\b(u)))([v,w]_V)\\&=&\rho(\a(\mathcal{H}(u)))([v,w]_V)\\& \overset{\eqref{actLie}}{=}&[\rho(\mathcal{H}(u))v,\b(w)]_V+[\b(v),\rho(\mathcal{H}(u))w]_V\\&=&[\{u,v\}_\mathcal{H},\b(w)]_V+[\b(v),\{v,w\}_\mathcal{H}]_V.
    \end{eqnarray*}
    On the other hand, we have
    \begin{eqnarray*}
        \{[u,v]_V,\b(w)\}_{\mathcal{H}}&=&\rho(\mathcal{H}([u,v]_V))(\b(w))\\&=&\rho([\mathcal{H}u,\mathcal{H}v])(\b(w))\\&=&\rho(\mathcal{H}(\rho(\mathcal{H}u)v))(\b(w))\\&=&\{\{u,v\}_\mathcal{H},\b(w)\}_\mathcal{H}
    \end{eqnarray*}
    Then,  the Eqs \eqref{eq:trileib1} and \eqref{eq:trileib2} are satisfied. Thus, the result follows.
\end{proof}
\subsection{Construction of Hom-Jordan trialgebras}
In the following, we introduce the notion of an action of Hom-Jordan algebra on a vector space $V$.
\begin{defi}
An  action  of a Hom-Jordan algebra $(A,\o,\a)$ on a Hom-Jordan algebra  $(V,\ast, \beta)$
is a representation  $\r: A \to gl(V)$ of $A$ such that
 for any $x,y \in A$ $u,v,w\in V$
 {\small\begin{align}
& \pi(\alpha(x))(\b(v))\ast(\b(v)\ast \b(v)) = (\pi(\alpha( x))(v\ast v))\ast \b^2(v),\label{ActHJ1}
\\
 &\nonumber (\pi(  x)(u)\ast \beta(w))\ast \beta^2(v)
+ (\pi(x)(v)\ast \b(w))\ast  \b^2(u)
+ \pi(\a^2(x))((u\ast v)\ast \b(w))   \\
 =& (\pi(\a(x))(u\ast v))\ast \b^2(w)
 + (\pi(\a(x))(u\ast w))\ast \b^2(v)
 + (\pi(\a(x))(v\ast w))\ast \b^2(u) ,
\\
  &(\pi(\a(y))(\pi(x)u))\ast\b^2(v)+ (\pi(\a(y))(\pi(x)v))\ast\b^2(u)+\pi(\a^2(x))(\pi(\a(y))(u\ast v))\nonumber  \\
 = &(\pi(x\o y)(\b(v)))\ast\b^2(u)+\pi(\a^2(x))(\pi(y)u\ast\b(v))+\pi(\a^2(y))(\pi(x)u\ast\b(v)) ,
\\&\nonumber
\pi(\a(y)u)\ast \pi(\a(x)v)+\pi(\a(x)u)\ast\pi(\a(y))v+\pi(\a(x)\o\a(y))(\b(u)\ast\b(v))   \\
 =& (\pi(x\o y)(\b(v)))\ast\b^2(u)+\pi(\a^2(x))(\pi(y)u\ast\b(v))+\pi(\a^2(y))(\pi(x)u\ast\b(v)).
\end{align}}
We denote it by $(V,\pi,\beta,\ast)$.
 \end{defi}
A \textsf{HJor-Act} pair  is a Hom-Jordan algebra $(A,\o,\a)$ with an action $(V,\pi,\b,\ast)$
and refer to it with the pair $(A,V)$.
 \begin{pro}
     Note that a tuple $(V,\pi,\b,\ast)$ is an action of a Hom-Jordan algebra $(A,\o,\a)$ if and only $A\oplus V$ carries a Hom-Jordan algebra structure with product and twisting map given by
     \begin{align*}
&(x + u) \o_{A \oplus V}  (y + v )=  x \o  y + \pi(x)v + \pi(y )u +u \ast  v, \\
&(\alpha+\b)(x+u)=\alpha(x)+\b(u)
\end{align*}
 for all  $x,y \in A$ and $ u,v \in B,$ which  is called the semi-direct product of $A$ with $V$.
 \end{pro}
In \cite{trijordan}, the authors introduce the notion of (right) Jordan trialgebra which is triple $(A,\o,\bullet)$ where $(A,\o)$ is a Jordan algebra and $(A,\bullet)$ is a Jordan dialgebra,
satisfying $4$ polynomial identities. Next, we give the definition of left Hom-Jordan trialgebra as a generalization of left Jordan dialgebras.
\emptycomment{
\begin{equation}
(x_1\o x_1)\o(x_1\bullet x_2) = ((x_1\o x_1)\bullet x_2)\o x_1. \label{JTL0}
\end{equation}
\begin{multline}
  ((x_1\bullet  x_4)\o x_3)\o x_2
+ ((x_2\bullet x_4)\o x_3)\o  x_1
+ ((x_1\o x_2)\o x_3)\bullet  x_4   \\
 = ((x_1\o  x_2)\bullet x_4)\o  x_3
 + ((x_1\o  x_3)\bullet x_4)\o  x_2
 + ((x_2\o  x_3)\bullet x_4)\o  x_1 .\label{JTL1}
\end{multline}
\begin{multline}
  ((x_1\bullet  x_4)\bullet x_3)\o x_2
+ ((x_2\bullet x_4)\bullet x_3)\o  x_1
+ ((x_1\o x_2)\bullet x_3)\bullet  x_4   \\
 = (x_2\bullet  (x_3\bullet x_4))\o  x_1
 + ((x_1\bullet  x_3)\o x_2)\bullet  x_4
 + ((x_1\bullet  x_4)\o x_2)\bullet  x_3 .\label{JTL2}
\end{multline}
\begin{multline}
 ( x_1\bullet x_3)\o (x_2\bullet x_4)
+ (x_1\bullet x_4)\o(  x_2\bullet x_3)
+ (x_1\o x_2)\bullet (x_3\bullet x_4)   \\
 = (x_2\bullet ( x_3\bullet x_4))\o x_1
 + ((x_1\bullet x_3)\o x_2) \bullet x_4
 + ((x_1\bullet x_4)\o x_2)\bullet x_3 .\label{JTL3}
\end{multline}}
\begin{defi}
  A   \emph{Hom-Jordan trialgebra} is a vector  space $A$ with two bilinear operations $\o,\bullet:A\times A \to A$   and  a linear map $\alpha:A\to A$, where $(A,\o,\a)$ is a Hom-Jordan algebra and $(A,\bullet,\a)$ is a Hom-Jordan dialgebra,
satisfying the following   identities:
{\small\begin{align}
&(\alpha(x_1)\o \alpha(x_1))\o( \alpha(x_2)\bullet\alpha(x_1)) = (\alpha( x_2)\bullet(x_1\o x_1))\o \alpha^2(x_1)\label{HJTL0}
\\
 &\nonumber\\
 &\nonumber ((x_4\bullet  x_1)\o \alpha(x_3))\o \alpha^2(x_2)
+ ((x_4\bullet x_2)\o \alpha(x_3))\o  \alpha^2(x_1)
+\alpha^2(x_4)\bullet  ((x_1\o x_2)\o \alpha(x_3))    \\
 =& (\alpha(x_4)\bullet(x_1\o  x_2) )\o \alpha^2( x_3)
 + (\alpha(x_4)\bullet(x_1\o  x_3))\o  \alpha^2(x_2)
 + (\alpha(x_4)\bullet(x_2\o  x_3))\o  \alpha^2(x_1) \label{HJTL1}\\
 &\nonumber
\\
  &(\alpha(x_3)\bullet(x_4\bullet  x_1) )\o \alpha^2(x_2)
+ (\alpha(x_3)\bullet(x_4\bullet x_2))\o \alpha^2( x_1)
+ \alpha^2(  x_4)\bullet(\alpha(x_3)\bullet(x_1\o x_2))) \nonumber  \\
 = &(  (x_4\bullet x_3)\bullet\alpha(x_2))\o \alpha^2( x_1)
 +\alpha^2(x_4)\bullet ((x_3\bullet  x_1)\o \alpha(x_2))
 +\alpha^2(x_3)\bullet ((x_4\bullet  x_1)\o \alpha(x_2))   \label{HJTL2}\\
 &\nonumber
\\&\nonumber
 (\alpha( x_3)\bullet\alpha( x_1))\o (\alpha(x_4)\bullet \alpha(x_2))
+ (\a(x_4)\bullet \a(x_1))\o(  \a(x_3)\bullet \a(x_2))
+  (\a(x_4)\bullet \a(x_3)) \bullet(\a(x_1)\o \a(x_2))  \\
 =& (( x_4\bullet x_3)\bullet\alpha(x_2))\o \alpha^2(x_1)
 + \a^2(x_4)\bullet((x_3\bullet x_1)\o \a(x_2))
 +\a^2(x_3)\bullet ((x_4\bullet x_1)\o \a(x_2))  .\label{HJTL3}
\end{align}}

When the twisting map $\alpha$ is the identity map, we recover the classical notion of Jordan trialgebra.
\end{defi}

\begin{defi} Let $(A,V)$ be a \textsf{HJor-Act} pair.
    A linear map $\mathcal H: V \to A$ is called a homomorphic relative averaging operator on $A$ with respect to $(V,\pi,\beta,\ast)$ if it's a relative averaging operator and a morphism, that is
      \begin{align}
  & \mathcal{H}(u\ast v)=\mathcal H(u)\o \mathcal H(v),\ \forall u,v \in V.
      \end{align}
\end{defi}
\begin{pro}
Let $(A,V)$ be a  \textsf{HJor-Act} pair. Then
  $(A\oplus V,\o_{A\oplus V},\bullet_{A\oplus V},\ \alpha_{A\oplus V})$ is Hom-Jordan trialgebra, called hemisemi-direct product Hom-Jordan trialgebra and denoted by $A\oplus_{TriJor} V$, where
  \begin{align*}
      (x+u)\o_{A\oplus V}(y+v)&=x\o y+u\ast v,\\ (x+u)\bullet_{A\oplus V}(y+v)&=x\o y+\pi(x)v.
  \end{align*}
\end{pro}
\begin{proof} For any $x_1,x_2\in A$ and $u_1,u_2\in V$ we have
\begin{align*}
&\big(\alpha_{A\oplus V}(x_1+u_1)\o_{A\oplus V} \alpha_{A\oplus V}(x_1+u_1)\big)\o_{A\oplus V}\big( \alpha_{A\oplus V}(x_2+u_2)\bullet_{A\oplus V}\alpha_{A\oplus V}(x_1+u_1)\big)\\=& \big((\a(x_1)\o\a(x_1)+\b(u_1)\ast \b(u_1)\big)\o_{A\oplus V}\big(\a(x_2)\bullet\a(x_1)+\pi(x_2)\b(u_1)\big)\\=&(\a(x_1)\o\a(x_1))\o(\a(x_2)\bullet\a(x_1))+(\b(u_1)\ast \b(u_1))\ast \pi(\a(x_2))\b(u_1)\\=&(\a(x_1)\o\a(x_1))\o(\a(x_2)\bullet\a(x_1))+\pi(\a(x_2))\b(u_1)\ast(\b(u_1)\ast \b(u_1))
\end{align*}
On the other hand, we have
\begin{align*}
   &\big(\alpha_{A\oplus V}( x_2+u_2)\bullet_{A\oplus V}((x_1+u_1)\o_{A\oplus V}(x_1+u_1))\big)\o_{A\oplus V}\alpha_{A\oplus V}^2(x_1+u_1)\\=&\big(\a(x_2)\bullet(x_1\o x_1)+\pi(\a(x_2))(\b(u_1)\ast \b(u_1))\big)\o_{A\oplus V}(\a^2(x_1)+\b^2(u_1))\\=&\big(\a(x_2)\bullet(x_1\o x_1)\big)\o\a^2(x_1)+\big(\pi(\a(x_2))(\b(u_1)\ast\b(u_1))\big)\ast \b^2(u_1)
\end{align*}
Thus it follows from Eq. \eqref{ActHJ1}, that the Eq. \eqref{HJTL0} is satisfied. Analogously, by a same calculation, one can show that the Eqs. \eqref{HJTL1}, \eqref{HJTL2} and \eqref{HJTL3} are satisfied.
\end{proof}
In the following result, we characterize homomorphic relative averaging operator on Hom-Jordan algebras by their graphs.
\begin{thm}
    Let $(A,V)$ be a \textsf{HJor-Act} pair, a linear map $\mathcal{H}:V\to A$ is a homomorphic relative averaging operator on the Hom-Jordan algebra $(A,\o,\alpha)$ if and only if  the graph $Gr(\mathcal{H}) = \{\mathcal{H}u+u\ |\ u \in V\}$ is a Hom-Jordan subtrialgebra of the hemisemi-direct product Hom-Jordan trialgebra algebra $A\oplus_{TriJor} V$.
\end{thm}
\begin{proof}
Let $\mathcal{H}:V\to A$ be a linear map. For any $u,v\in V$ we have
\begin{eqnarray*}
    (\mathcal{H}(u)+u)\bullet_{A\oplus V}(\mathcal{H}(v)+v)&=&\mathcal{H}(u)\o\mathcal{H}(v)+\pi(\mathcal{H}(u))v,\\
(\mathcal{H}(u)+u)\o_{A\oplus V}(\mathcal{H}(v)+v)&=&\mathcal{H}(u)\o\mathcal{H}(v)+u\ast v.
\end{eqnarray*}
Therefore, the graph $Gr({\mathcal H}) = \{{\mathcal H}u+u\ |\ u \in V\}$ is a  subalgebra of the hemisemi-direct product $A\oplus_{TriJor} V$ if and only if $\mathcal{H}$ satisfies $\mathcal{H}(u)\o \mathcal{H}(v)=\mathcal{H}(\pi(\mathcal{H}u)v)=\mathcal{H}(u\ast v)$, wich implies that $\mathcal{H}$ is a homomorphic relative averaging operator on \textsf{HJor-Act} pair $(A,V)$.
\end{proof}
\begin{pro}
    Let $(A,V)$ be a \textsf{HJor-Act} pair. A linear map $\mathcal H: V \to A$ is a homomorphic relative averaging operator on $V$ over $A$ if and only if the map
$$N_{\mathcal H}: (A \oplus V) \to (A \oplus V),\ x+u \mapsto \mathcal H(u)$$
is a Nijenhuis operator on the  hemisemi-direct product $A \oplus_{TriJor} V$.
\end{pro}
\begin{thm}
    Let $\mathcal{H}:V\to A$ be a homomorphic relative averaging operator on the \textsf{HJor-Act} pair $(A,V)$. Then $(V,\bullet_\mathcal{H},\ast,\b)$ is a  Hom-Jordan trialgebra, where
    \begin{equation}
     u\bullet_\mathcal{H} v=\pi(\mathcal{H}(u))v, \quad \forall u,v\in V.
    \end{equation}
\end{thm}
\begin{proof}
    It follows from Theorem \ref{JortodiJor} that $(V,\bullet_\mathcal{H},\b)$ is a Hom-Jordan dialgebra. On the other hand, for any $u_1,u_2,u_3,u_4\in V$ we have
    \small\begin{align*}
        &(\b(u_1)\ast \b(u_1))\ast( \b(u_2)\bullet_\mathcal{H}\b(u_1)) - (\b(u_2)\bullet_\mathcal{H}(u_1\ast u_1))\ast \b^2(u_1)\\= &(\b(u_2)\bullet_\mathcal{H} \b(u_1))\ast (\b(u_1)\ast \b(u_1))- (\b(u_2)\bullet_\mathcal{H}(u_1\ast u_1))\ast \b^2(u_1)\\= &(\pi(\mathcal{H}(\b(u_2))\b(u_1))\ast (\b(u_1)\ast \b(u_1))- (\pi(\mathcal{H}(\b(u_2))(u_1\ast u_1))\ast \b^2(u_1)\\= &(\pi(\a(\mathcal{H}(u_2))\b(u_1))\ast (\b(u_1)\ast \b(u_1))- (\pi(\a(\mathcal{H}(u_2))(u_1\ast u_1))\ast \b^2(u_1)\\ \overset{\eqref{ActHJ1}}{=}&0.
    \end{align*}
    Then, the Eq. \eqref{HJTL0} is satisfied. Similarly, one can check the Eqs \eqref{HJTL1}, \eqref{HJTL2} and \eqref{HJTL3}.
\end{proof}
According to Theorem $\ref{diasstodijor}$, we have the following result
\begin{thm}\label{triasstotrijor}

 Let $\mathcal{H}:V\to A$ be a homomorphic relative averaging operator on a \textsf{HAss-Act} pair $(A,V)$. Then $\mathcal{H}$ is a homomorphic relative averaging operator on the associated Hom-Jordan algebra  $(A,\o=\c+\c^{op},\alpha)$, with respect to an action $(V,\pi=l+r,\b,\ast=\c_V+\c_V ^{op})$.

 Moreover, if $(V,\lprod,\rprod,\mprod,\b)$ is the Hom-associative trialgebra associated to the Hom-associative algebra $(A,\c,\alpha)$. Then $(V,\bullet,\ast,\b)$ is a Hom-Jordan trialgebra, where $"\bullet"$ is the anti-dicommutator defined in Eq. \eqref{antidicomm}.
\end{thm}

More generally, we have the following diagram which summarize and extend the admissibility
of Hom-Lie and Hom-Jordan algebras in the level of Hom-dialgebras and Hom-trialgebras.

\begin{equation*}
\xymatrix{\text{Hom-TriLeib}\ar@{<-}[d]_{(\{\c,\c\},\a)\to(\{\c,\c\},0,\a)} & &\ar@{->}_{\text{Thm}\ \ref{triasstotriLeib}}[ll] \text{Hom-triass}\ar@{<-}[d]_{(\vdash,\dashv,\a)\to(\vdash,\dashv,0,\a)} & &\text{Hom-triJor} \ar@{<-}_{\text{Thm}\ \ref{triasstotrijor}}[ll]\ar@{<-}_{(\bullet,\a)\to(\bullet,0,\a)}[d]\\\text{Hom-Leib}\ar@{<-}[d]_{\text{Thm}\ \ref{LietoLeibniz}} & &\ar@{->}_{\text{Prop}\ \ref{asstoLeibniz}}[ll] \text{Hom-diass}\ar@{<-}[d]_{\text{Thm}\ \ref{asstodiass}} & &\text{Hom-diJor} \ar@{<-}_{\text{Thm}\ \ref{diasstodijor}}[ll]\ar@{<-}_{\text{Thm}\ \ref{JortodiJor}}[d] \\
\text{Hom-Lie} & & \text{Hom-ass} \ar@{->}_{-}[ll] & & \text{Hom-Jordan} \ar@{<-}_{+}[ll]}
\end{equation*}


\begin{thebibliography}{999}


\bibitem{trijordan}
F. Bagherzadeh, M. Bremner, S. Madariaga, Jordan trialgebras and post-Jordan algebras, Journal of Algebra, 486
(2017), 360–395

\bibitem{bai-imrn} C. Bai, O. Bellier, L. Guo and X. Ni, Splitting of operations, Manin products and Rota-Baxter operators,
{\em Int. Math. Res. Not.} 2013,  no. 3 (2013), 485-524.


\bibitem{Berg2}
E. A. Bergshoeff, M. de Roo and O. Hohm, Multiple M2-branes and the embedding tensor. Classical Quantum
Gravity 25 (2008), 142001, 10 pp.

\bibitem{birkhoff} G. Birkhoff, Moyennes de fonctions born\'{e}es, Coil. Internat. Centre Nat. Recherthe Sci. (Paris), {\em Alg\'{e}bre Th\'{e}forie Nombres} 24 (1949), 149-153.

\bibitem{bon} R. Bonezzi and O. Hohm, Leibniz gauge theories and infinity structures, {\em Commun. Math. Phys.} 377 (2020), 2027-2077.

\bibitem{brain} B. Brainerd, On the structure of averaging operators, {\em J. Math. Anal.} 5 (1962), 347-377.



\bibitem{cao} W. Cao, An algebraic study of averaging operators, Ph.D. Thesis, Rutgers University at Newark (2000).


\bibitem{das} A. Das,   Controlling structures, deformations and homotopy theory for averaging algebras. arXiv preprint arXiv:2303.17798. (2023).


\bibitem{dasmakhlouf}
A. Das, A. Makhlouf, (2023). Embedding tensors on Hom-Lie algebras. arXiv preprint arXiv:2304.04178.

\bibitem{DasMandal}
A. Das and R. Mandal,  Averaging algebras of any nonzero weight. arXiv preprint arXiv:2304.12593.(2023). 

\bibitem{Fernandez}
J. J. Fernandez-Melgarejo, T. Ortin and E. Torrente-Lujan, Maximal nine dimensional supergravity, general
gaugings and the embedding tensor. Fortschr. Phys. 60 (2012), 1012-1018

\bibitem{gam-mill} J. L. B. Gamlen and J.B. Miller, Averaging operators and Reynolds operators on Banach algebras II. spectral
properties of averaging operators, J. Math. Anal. Appl 23 (1968), 183-197.

\bibitem{Hartwig}
J. T. Hartwig, D. Larsson, and S.D. Silvestrov, Deformations of Lie algebras using
$\sigma$-derivations, J. Algebra 295 (2006) 314–361.

\bibitem{kampe}  J. Kamp\'{e} de F\'{e}riet, L'etat actuel du probl\'{e}me de la turbulaence (I and II), {\em La Sci. A\'{e}rienne} 3 (1934) 9-34, 4 (1935), 12-52.


\bibitem{kelly} J. L. Kelly, Averaging operators on $C_\infty (X)$, {\em Illinois J. Math.} 2 (1958), 214-223.








\bibitem{Kolesnikov}
P. S. Kolesnikov, (2008). Varieties of dialgebras and conformal algebras. Siberian Mathematical Journal, 49(2), 257-272.

\bibitem{Kotov}
A. Kotov and T. Strobl, The embedding tensor, Leibniz-Loday algebras, and their higher gauge theories. Comm.
Math. Phys. 376 (2020), 235-258.

\bibitem{Larson}
D. Larsson and S. D. Silvestrov, Quasi-Hom-Lie algebras, Central Extensions and
2-cocycle-like identities. J. of Algebra, 288 (2005), 321–344.


\bibitem{lavau}
S. Lavau, Tensor hierarchies and Leibniz algebras. J. Geom. Phys. 144 (2019), 147-189.

\bibitem{Mainellis}
E. Mainellis, (2022). Classification of Low-dimensional Complex triassociative algebras. arXiv preprint arXiv:2209.04351.

\bibitem{Makhlouf}
 A. Makhlouf, Hom-alternative algebras and Hom-Jordan algebras, arXiv: 0909.0326.

\bibitem{Makhlouf1}
A. Makhlouf, (2012). Hom-dendriform algebras and Rota-Baxter Hom-algebras. In Operads and Universal Algebra (pp. 147-171).

\bibitem{MakSil}
A. Makhlouf and S. Silvestrov, Hom-algebra structures, J. Gen. Lie Theory Appl. 2
(2008) 51–64.


\bibitem{mill} J. B. Miller, Averaging and Reynolds operators on Banach algebra I, Representation by derivation and
antiderivations, {\em J. Math. Anal. Appl.} 14 (1966), 527-548.

\bibitem{moy} S.-T. C. Moy, Characterizations of conditional expectation as a transformation on function spaces, {\em Pacific J. Math.} 4 (1954), 47-63.

\bibitem{Nicolai}
H. Nicolai and H. Samtleben, Maximal gauged supergravity in three dimensions. Phys. Rev. Lett. 86 (2001),
1686-1689.

\bibitem{pei-rep} J. Pei, C. Bai, L. Guo and X. Ni, Replicators, Manin white product of binary operads and average operators, {\em New Trends in Algebras and Combinatorics}, pp. 317-353 (2020).

\bibitem{pei-rep2} J. Pei, C. Bai, L. Guo and X. Ni, Disuccessors and duplicators of operads, Manin products and operators,
In ``Symmetries and Groups in Contemporary Physics'', {\em Nankai Series in Pure, Applied Mathematics and
Theoretical Physics} 11 (2013), 191-196.

\bibitem{pei-guo} J. Pei and L. Guo, Averaging algebras, Schr\"{o}der numbers, rooted trees and operads, {\em J. Algebr. Comb.} 42 (2015), 73-109.

\bibitem{tangsheng}
Y. Sheng, R. Tang (2023). Nonabelian embedding tensors. Letters in Mathematical Physics, 113(1), 14.

\bibitem{rey} O. Reynolds, On the dynamic theory of incompressible viscous fluids, {\em Phil. Trans. Roy. Soc. A} 136
(1895), 123-164.

\bibitem{sheng-embedd} Y. Sheng, R. Tang and C. Zhu, The controlling $L_\infty$-algebra, cohomology and homotopy of embedding tensors and Lie-Leibniz triples, {\em Commun. Math. Phys.} 386 (2021), 269-304.

\bibitem{strobl} T. Strobl, Leibniz-Yang-Mills gauge theories and the $2$-Higgs mechanism, {\em Phys. Rev. D} 99 (2019), 115026.

\bibitem{Yau}
D. Yau, Hom-Maltsev, Hom-alternative and Hom-Jordan algebras. International
Electronic Journal of algebras, 11 (2012), 177–217.




































































\end{thebibliography}
\end{document}